\documentclass{article}

\usepackage{amsmath}
\usepackage{amssymb}
\usepackage{amsthm}
\usepackage[latin1]{inputenc}
\usepackage[english]{babel}
\usepackage[T1]{fontenc}
\usepackage{tikz}
\usepackage{xcolor}
\usepackage{comment}

\usepackage{latexsym,color,graphicx,enumerate}
\usepackage{hyperref}

\DeclareMathOperator*{\Res}{Res}

\newtheorem{thm}{Theorem}[section]
\newtheorem{prop}[thm]{Proposition}
\newtheorem{coro}[thm]{Corollary}
\newtheorem{lemma}[thm]{Lemma}
\newtheorem{rem}[thm]{Remark}

\newtheorem{nota}{Notation}[section]

\makeatletter

\newcommand*{\plim}[1][]{%
	\if\relax\detokenize{#1}\relax
	\def\next{\qopname\relax m{lim}}%
	\else
	\def\next{\qopname\newmcodes@ m{#1-lim}}%
	\fi
	\next
}

\newcommand*{\psum}[1][]{%
	\DOTSB
	\if\relax\detokenize{#1}\relax\else
	\operatorname{#1-}\mkern-\thinmuskip
	\fi
	\sum@\slimits@
}

\makeatother

\newcommand{\R}{\mathbb{R}}             
\newcommand{\N}{\mathbb{N}}             
\newcommand{\C}{\mathbb{C}}             

\newcommand{\B}{\mathcal{B}}            

\newcommand {\clb}{\color{blue}}
\newcommand {\clr} {\color{red}}
\newcommand {\clm} {\color{magenta}}

\numberwithin{equation}{section}

\title{Inverse Spectral Analysis of Singular Radial AKNS Operators}

\author{Damien Gobin, Beno\^it Gr\'ebert, Bernard Helffer and Fran{\c{c}}ois Nicoleau}

\date{}

\begin{document}
	
	\maketitle
	
\begin{abstract}
	We study an inverse spectral problem for singular AKNS operators based on
	spectral data associated with two distinct values of the effective  angular momentum
	parameter $\kappa\,$.
	Our main focus is the local inverse problem near the zero potential.
	For the pairs $(\kappa_1,\kappa_2)=(0,1)$, $(1,2)$ and $(0,3)\,$, we establish local
	uniqueness.
	For $(0,2)\,$, we prove that the Fr\'echet differential of the spectral map
	at the origin is injective, while the question whether its range is closed
	remains open.
\end{abstract}

	\section{Introduction}

	Inverse spectral theory for one-dimensional singular differential operators
	arises naturally in the analysis of radially symmetric quantum systems.
	The classical example is the radial Schr\"odinger operator, obtained from the
	three-dimensional Schr\"odinger equation after separation of variables in
	spherical coordinates. For a real-valued square-integrable potential
	$q \in L^2(0,1)\,$, the radial equation takes the form
	\begin{equation}\label{eq:Schr}
		H_\ell(q)u
		:=
		-\frac{d^2u}{dx^2}
		+\frac{\ell(\ell+1)}{x^2}u
		+ q(x)u
		= \lambda u\,,
		\qquad x\in(0,1)\,,\quad \ell\in\mathbb N\,.
	\end{equation}
	The regularity condition is $u(x)=O(x^{\ell+1})$ as $x\to0 \,$, together with the
	Dirichlet boundary condition $u(1)=0\,$.
	For each  angular momentum~$\ell$, this defines a self-adjoint operator on
	$L^2(0,1)$ with discrete, simple, real spectrum.

	\vspace{0.2cm}\noindent
	The inverse spectral problem consists in determining the potential $q(x)$ from
	its spectral data. A classical result due to  P\"oschel--Trubowitz \cite{PT}, Carlson \cite{Carlson97}, Guillot--Ralston \cite{GR} and 
	Zhornitskaya--Serov \cite{zs} asserts that the Dirichlet
	spectrum together with suitable norming constants forms a real-analytic
	coordinate system on $L^2(0,1)$ for each fixed~$\ell\,$.
	Recently, in the context of radial Schr\"odinger operators with distinct angular
	momenta, we proved in~\cite{GGHN2025} that the potential is uniquely determined
	by the Dirichlet spectra corresponding to infinitely many values of~$\ell$
	satisfying a M\"untz-type condition, and that in a neighborhood of the zero
	potential, knowing two spectra (for the pairs $(\ell_1,\ell_2)=(0,1),\ (1,2)$ or $(0,3)$)
	already implies uniqueness. These results rely on the explicit structure of the eigenfunctions in terms of Bessel functions, specifically of the form
	\begin{equation}\label{eq:SchrEigen}
		u_{\ell,n}(x)
		= c_{\ell,n}\, x^{1/2} J_\nu(j_{\nu,n} x)\,,
		\qquad
		\nu=\ell+\tfrac12\,,
	\end{equation}
	together with delicate completeness properties of the squared eigenfunctions,
	following earlier works of Rundell and Sacks \cite{RuSa01}.
	
\vspace{0.2cm}\noindent
The aim of the present work is to investigate the analog of this spectral problem for singular radial AKNS operators.
The mathematical model considered here arises from certain physical models,
whose derivation is briefly outlined in the appendix.
After separation of variables, one is led to a family of singular radial AKNS operators parameterized by what we call an \emph{effective angular momentum parameter} $\kappa \in \mathbb{Z}$. We emphasize that the parameter $\kappa$ does not, in general, correspond to a genuine angular momentum, in contrast with the Schr\"odinger case. In the 3D Dirac framework, $-\kappa$ arises as an eigenvalue of the spin--orbit operator $K$, whereas in the 2D model it can be interpreted as an effective angular momentum (see (7.97) in \cite{Th92} and Appendix A of the present paper).

\vspace{0.2cm}\noindent
The associated singular AKNS operator is
\begin{equation}\label{eq:AKNS}
	H_\kappa(V)\,Z = H_\kappa(p,q) \, Z
	=
	\begin{pmatrix}
		0 & -1\\[0.2em]
		1 & 0
	\end{pmatrix} Z'
	+
	\begin{pmatrix}
		0 & -\dfrac{\kappa}{x}\\[0.4em]
		-\dfrac{\kappa}{x} & 0
	\end{pmatrix} Z
	+
	V(x)\,Z\,,
	\qquad
	Z=(Z_1,Z_2)^{\mathsf T}\,.
\end{equation}
Here the potential matrix $V$ is given by
\begin{equation}\label{eq:V}
	V(x)
	=
	\begin{pmatrix}
		-q(x) & p(x)\\[0.2em]
		p(x) & q(x)
	\end{pmatrix}\,,
	\qquad
	p,q \in L_\R^2(0,1)\,.
\end{equation}
We impose the following boundary conditions. Let $(\theta_1,\theta_2)\in\mathbb{R}^2\,$.
\begin{itemize}
	\item When $\kappa=0$\,,
	\begin{equation}\label{eq:MITa}
		Z(0)\cdot u_{\theta_1} = 0\,,
		\qquad
		Z(1)\cdot u_{\theta_2} = 0\,,
		\qquad
		u_{\theta_j}=\binom{\sin\theta_j}{\cos\theta_j}\,.
	\end{equation}
	
	\item When $\kappa\neq 0\,$,
	\begin{equation}\label{eq:MITb}
		Z(1)\cdot u_{\theta_2} = 0 \,.
	\end{equation}
\end{itemize}

\medskip\noindent
The AKNS system enjoys symmetries associated with the Pauli matrices
\begin{equation}\label{paulimatrices}
\sigma_1=\begin{pmatrix}0&1\\[0.2em]1&0\end{pmatrix},\quad
\sigma_2=\begin{pmatrix}0&-i\\[0.2em]i&0\end{pmatrix},\quad
\sigma_3=\begin{pmatrix}1&0\\[0.2em]0&-1\end{pmatrix}.
\end{equation}
A direct computation shows that
\begin{align}
\sigma_1 H_\kappa(p,q)\sigma_1 &= -\,H_{-\kappa}(-p,q), \label{eq:sym-sigma1}\\
\sigma_2 H_\kappa(p,q)\sigma_2 &= \,H_{-\kappa}(-p,-q), \label{eq:sym-sigma2}\\
\sigma_3 H_\kappa(p,q)\sigma_3 &= -\,H_{\kappa}(p,-q). \label{eq:sym-sigma3}
\end{align}
In particular, if $Z$ solves $H_\kappa(p,q)Z=\lambda Z$, then
$\sigma_1 Z$, $\sigma_2 Z$ and $\sigma_3 Z$ solve the corresponding
transformed systems with the same or opposite spectral parameter
according to \eqref{eq:sym-sigma1}--\eqref{eq:sym-sigma3}.

\medskip\noindent
In the case $p=0$, the $\sigma_1$-symmetry shows that the knowledge of
the spectra corresponding to $\kappa$ and $-\kappa$ (with the same
boundary condition $\theta_2=0$) is equivalent, up to a reindexation,
to the knowledge of the two spectra associated with $\theta_2=0$ and
$\theta_2=\pi/2$ for $H_\kappa$. Therefore, provided that the resulting sequences satisfy the technical 
interlacing property required in Theorem~1.1 of \cite{AHM}, the
potential $q$ is uniquely determined.

\medskip
\noindent
For these reasons, and other technical difficulties\footnote{It would actually be interesting to treat the general case, but, except for the case $\theta_2=\frac \pi 2$,  this would indeed be technically more involved.},  we restrict ourselves, in the present paper,  to the case $$ \kappa \geq 0 \mbox{ and } \theta_2=0 \,.$$
In the case $\kappa=0\,$, we also set $\theta_1=0\,$. As explained in the appendix, $\theta_2=0$  corresponds in the motivating Dirac radial model to the so-called Zig--Zag condition, while the MIT bag condition corresponds to $\theta_2=\pm \frac{\pi}{4}$\,.

\medskip
\noindent
The domain of the operator is then defined as follows:
\begin{equation}\label{eq:Domaina}
	D(H_0)
	=
	\Bigl\{
	Z=(Z_1,Z_2)\in L^2(0,1)^2 \;:\;
	H_0 Z \in L^2(0,1)^2\,,\;
	Z_2(0)=0\,,\;
	Z_2(1)=0
	\Bigr\}\,.
\end{equation}
For $\kappa >0$, we set
\begin{equation}\label{eq:Domainb}
	D(H_\kappa)
	=
	\Bigl\{
	Z=(Z_1,Z_2)\in L^2(0,1)^2 \;:\;
	H_\kappa Z \in L^2(0,1)^2\,,\;
	Z_2(1)=0
	\Bigr\}\,.
\end{equation}
As shown in~\cite{Serier2006}, this realizes a self-adjoint operator with purely
discrete and {\it simple} spectrum which can be written as a doubly infinite sequence 
\begin{equation}\label{eq:EigenSeq}
	\{\lambda_{\kappa,n}(p,q)\}_{n\in\mathbb Z}\,,
\end{equation}
ordered as
\begin{equation}\label{eq:Order}
	\cdots < \lambda_{\kappa,-2}(p,q)
	< \lambda_{\kappa,-1}(p,q)
	< \lambda_{\kappa,0}(p,q)
	< \lambda_{\kappa,1}(p,q)
	< \lambda_{\kappa,2}(p,q)
	< \cdots \, .
\end{equation}

\medskip\noindent
The labelling is uniquely determined by the asymptotic behavior (see \cite{Serier2006}, Theorem 3.1):
\begin{equation}\label{eq:AsymptAKNS}
	\lambda_{\kappa,n}(p,q)
	=
	\Bigl(n+\operatorname{sgn}(n)\,\tfrac{\kappa}{2}\Bigr)\,\pi
	+ \ell^2(n)\,,
	\mbox{ as }  |n|\to\infty\,,
\end{equation}
which governs both ends of the sequence. Here the notation
$\alpha_n=\beta_n+\ell^2(n)\,$, $n\in\mathbb{Z}\,$, means that the sequence
$(\alpha_n-\beta_n)_{n\in\mathbb{Z}}$ belongs to $\ell^2(\mathbb{Z})\,$.

\medskip\noindent
In addition to the eigenvalues, Serier introduced suitable norming
constants which, together with the spectrum, form a complete system of spectral
coordinates. More precisely, the combined data (eigenvalues and norming
constants) provide a locally stable parameterization of the potential and yield a
Borg--Levinson type uniqueness result (see~\cite{Serier2006}).

\medskip\noindent
However, these classical results rely crucially on the availability of norming
constants. From the physical and inverse point of view, such quantities are in
general \emph{not observable}. This naturally leads to a different and more
challenging question: can one determine the potential uniquely from spectral data
alone, \emph{without} any norming constants?

\medskip\noindent
As in the corresponding inverse problem for the radial Schr\"odinger operator,
the spectral data associated with a single effective angular momentum are not sufficient
to ensure uniqueness. This naturally leads to combining information from
\emph{at least two distinct effective angular momenta} $\kappa\,$.  More precisely, we consider the  spectra corresponding to
two distinct effective angular momenta $\kappa_1 \neq \kappa_2$ and study whether this purely
spectral information, without any norming constants, determines the potential.

\medskip\noindent
We now state our first main result,  which gives several cases where two spectra are sufficient
to recover the potential locally near the trivial configuration.

\begin{thm}[Local uniqueness for the pairs $(0,1)$, $(1,2)$ and $(0,3)$]
	\label{thm:main-AKNS-01}
	Let $(\kappa_1,\kappa_2)=(0,1)$, $(1,2)$ or $(0,3)\,$. Then the knowledge of the  spectra
	associated with the effective angular momenta $\kappa_1$ and $\kappa_2$ uniquely determines
	the potential $V=(p,q)\in L^2(0,1)\times L^2(0,1)$ in a neighborhood of
	the zero potential $V_0=0\,$.	
\end{thm}

\medskip\noindent
The proof of Theorem~\ref{thm:main-AKNS-01} relies on the analysis of the
Fr\'echet differential of the associated spectral map at the zero potential.
We show that this differential is injective with closed range, which yields the
desired local uniqueness result.

\medskip\noindent
We now briefly introduce this spectral map. Let
$\widetilde{\lambda}_{\kappa,n}(p,q)$ denote the renormalized eigenvalues,
implicitly defined by the asymptotic formula \eqref{eq:AsymptAKNS} and explicitly
given by
\[
\widetilde{\lambda}_{\kappa,n}(p,q)
=
\lambda_{\kappa,n}(p,q)
-
\Bigl(n + \operatorname{sgn}(n)\,\tfrac{\kappa}{2}\Bigr)\,\pi\,.
\]
We choose two distinct  effective angular momenta
$\kappa_1 \neq \kappa_2 \,$, which are fixed  integers and we consider the
associated spectral map
\[
\mathcal{S}_{\kappa_1,\kappa_2} \,:\,
L^2(0,1) \times L^2(0,1)
\longrightarrow
\ell_{\mathbb{R}}^2(\mathbb{Z}) \times \ell_{\mathbb{R}}^2(\mathbb{Z})\,,
\]
defined by
\begin{equation}\label{eq:spectral-map-two-ell}
	\mathcal{S}_{\kappa_1,\kappa_2}(p,q)
	=
	\left(
	\bigl(\widetilde{\lambda}_{\kappa_1,n}(p,q)\bigr)_{n\in\mathbb{Z}},
	\;
	\bigl(\widetilde{\lambda}_{\kappa_2,n}(p,q)\bigr)_{n\in\mathbb{Z}}
	\right)\,.
\end{equation}

\medskip\noindent
We now state the second main result of this paper,  which discusses the injectivity 
of the Fr\'echet differential of the spectral map at the zero potential for three pairs of 
 effective angular momenta.

\begin{thm}[Behavior of the differential of the spectral map]
	\label{thm:differential-behavior}
	Let $\kappa_1 \neq \kappa_2$ be two distinct  integers and consider the
	spectral map
	\[
	\mathcal{S}_{\kappa_1,\kappa_2} \,:\,
	L^2(0,1) \times L^2(0,1)
	\longrightarrow
	\ell_{\mathbb{R}}^2(\mathbb{Z}) \times \ell_{\mathbb{R}}^2(\mathbb{Z})\,.
	\]
	Then, at the zero potential $V=0\,$, the Fr\'echet differential of
	$\mathcal{S}_{\kappa_1,\kappa_2}$ satisfies:
	
	\begin{itemize}
		\item For $(\kappa_1,\kappa_2)=(0,1)\,$, $(1,2)$ or $(0,3)\,$, the differential is injective and has closed range.
		
		\item For $(\kappa_1,\kappa_2)=(0,2)\,$, the differential is injective.
		
	\end{itemize}
\end{thm}

\medskip\noindent
The proof of Theorem~\ref{thm:differential-behavior} is given in
Sections~6--7. The case $(0,2)$ remains open, as we have not been able
to prove that the differential has closed range.
Theorem~\ref{thm:main-AKNS-01} then follows from the local injectivity result stated in Proposition~\ref{prop:local-injectivity}.

\medskip\noindent
In the appendix, we will describe how these questions arise in the analysis of inverse spectral problems for the radial Dirac operator (with possible addition of an Aharonov--Bohm potential) in dimension two and three and present some remaining open questions.

\section{Eigenvalue analysis in the unperturbed case $V=0$}

In this section we analyze the spectral problem in the unperturbed case $V=0$,
which serves as the reference configuration for the perturbative and inverse
analysis developed later.

\medskip\noindent
When $V=0\,$, the AKNS operator reduces to the first-order matrix equation
\begin{equation}\label{ODEmatrix}
\begin{pmatrix}
0 & -1\vphantom{\dfrac{\kappa}{x}}\\[0.3em]
1 & 0
\end{pmatrix} Z'(x)
+
\begin{pmatrix}
0 & -\dfrac{\kappa}{x}\\[0.5em]
-\dfrac{\kappa}{x} & 0
\end{pmatrix} Z(x)
=\lambda Z(x),
\qquad x\in(0,1),
\end{equation}
with the boundary condition 
\begin{equation}\label{BC-V0}
	Z_2(0)=0\,,\; Z_2(1)=0 \quad \text{for } \kappa = 0\,,
	\qquad
	Z_2(1)=0\,, \quad \text{for } \kappa \neq 0\, .
\end{equation}

\subsection{The case $\lambda=0$}

The case $\lambda=0$ requires a specific discussion.  
Setting $\lambda=0$ in the unperturbed Dirac equation yields
\[
Z_1'=\frac{\kappa}{x}\,Z_1\,,
\qquad
Z_2'=-\frac{\kappa}{x}\,Z_2\, .
\]
Hence the general solutions are
\[
Z_1(x)=C_1\,x^{\kappa}\,,
\qquad
Z_2(x)=C_2\,x^{-\kappa}\,.
\]
For $\kappa=0\,$, the boundary condition $Z_2(0)=0$ forces $C_2=0$.
For $\kappa\geq 1\,$, the condition $Z_2\in L^2(0,1)$ near $x=0$ again forces $C_2=0\,$. so that
\[
Z(x)=\binom{C_1\,x^{\kappa}}{0}\,.
\]
The  boundary condition at $x=1$ is satisfied, and therefore $0$ is an
eigenvalue of $H_\kappa(0)\,$.
The associated eigenfunction is thus given by
\[
Z_{\kappa,0}^{(0)}(x)
=
c_{\kappa,0}\,
\binom{x^{\kappa}}{0}\,,
\qquad
c_{\kappa,0}>0\,,
\]
and the normalization condition
\(
\|Z_{\kappa,0}^{(0)}\|_{L^2(0,1)^2}=1\,,
\)
yields
\[
c_{\kappa,0}=\sqrt{2\kappa+1}\,,
\qquad\text{so that}\qquad
Z_{\kappa,0}^{(0)}(x)=\sqrt{2\kappa+1}\binom{x^\kappa}{0}\,.
\]

\subsection{The case $\lambda \not=0$}

We now consider the case $\lambda \neq 0$, for which the system \eqref{ODEmatrix} decouples into two scalar equations for $Z_1$ and $Z_2$:
	\begin{equation}\label{ODE-Z1Z2}
		\left\{
		\begin{aligned}
			&Z_1''(x)+\left(\lambda^2-\frac{\kappa(\kappa-1)}{x^2}\right)Z_1(x)=0\,,\\[0.4em]
			&Z_2''(x)+\left(\lambda^2-\frac{\kappa(\kappa+1)}{x^2}\right)Z_2(x)=0\,.
		\end{aligned}
		\right.
	\end{equation}
	We set $\nu=\kappa+\tfrac12$.
	After the standard substitution $Z_j(x)=\sqrt{x}\,u_j(\lambda x)\,$, the system reduces to Bessel equations of orders $\nu-1$ and $\nu$.
	Accordingly, a fundamental system of solutions is given by
	\begin{equation}\label{eq:Fundamental}
		Z^{(0)}(x,\lambda)
		=
		\sqrt{\frac{\pi \lambda x}{2}}\,
		\begin{pmatrix}
			J_{\nu-1}(\lambda x)\\[0.4em]
			-\,J_{\nu}(\lambda x)
		\end{pmatrix}\,,
		\qquad
		W^{(0)}(x,\lambda)
		=
		\sqrt{\frac{\pi \lambda x}{2}}\,
		\begin{pmatrix}
			-\,Y_{\nu-1}(\lambda x)\\[0.4em]
			Y_{\nu}(\lambda x)
		\end{pmatrix}\,.
	\end{equation}
where $J_\nu$ and $Y_\nu$ denote the Bessel functions of the first and
second kinds (see \cite{Lebedev72}).  We recall that for $z\in\mathbb{C}\setminus(-\infty,0]$ and any real or complex parameter
$\nu\,$, one has
\begin{equation}\label{BesselJ}
J_{\nu}(z)
=\Bigl(\frac{z}{2}\Bigr)^{\nu}
\sum_{k=0}^{\infty}
(-1)^{k}\,
\frac{\bigl(\tfrac{z^{2}}{4}\bigr)^{k}}{k!\,\Gamma(\nu+k+1)}\,.
\end{equation}
For $\nu\notin\mathbb{Z}\,$, the function $Y_\nu$ is given by \footnote{For an integer $n$, the function $Y_n(z)$ is defined by the limit
\(
Y_n(z)=\lim_{\nu\to n} Y_\nu(z)\,.
\)}

\begin{equation}\label{BesselY}
Y_{\nu}(z)
=
\frac{
   J_{\nu}(z)\,\cos(\nu\pi)\;-\; J_{-\nu}(z)
}{
   \sin(\nu\pi)
}\,.
\end{equation}
We emphasize that for half-integer orders \(\nu\in \tfrac12+\mathbb{N}\), the function
\(\sqrt{z}\,J_\nu(z)\) is entire, whereas \(\sqrt{z}\,Y_\nu(z)\) is
generally meromorphic because of its singularity at \(z=0\).

\vspace{0.2cm}\noindent
As $x\to 0$\,, these functions satisfy the classical asymptotics
:
\begin{equation}\label{singularite}
J_\nu(x)\sim \frac{1}{\Gamma(\nu+1)}
\left(\frac{x}{2}\right)^\nu\,,
\qquad
Y_\nu(x)\sim -\,\frac{\Gamma(\nu)}{\pi}
\left(\frac{2}{x}\right)^\nu\,,
\qquad \nu>0\,.
\end{equation}
Hence $J_\nu$ is regular at the origin, whereas $Y_\nu$ is singular. In particular, among the two fundamental solutions in
\eqref{eq:Fundamental}, only $Z^{(0)}(x,\lambda)$ is square--integrable
near $x=0$ and satisfies the regularity condition.
For each eigenvalue $\lambda=\lambda_{\kappa,n}(0,0)$ of the unperturbed operator,
we define the associated eigenfunction by
\begin{equation}\label{eq:Fundamentalbis}
Z_{\kappa,n}^{(0)}(x)
=
c_{\kappa,n}
\begin{pmatrix}
Z_{1,\kappa,n}^{(0)}(x)\\[0.2em]
Z_{2,\kappa,n}^{(0)}(x)
\end{pmatrix}
=
c_{\kappa,n}\,
\sqrt{\lambda_{\kappa,n}(0,0) x}\,
\begin{pmatrix}
	J_{\nu-1}\!\bigl(\lambda_{\kappa,n}(0,0)\,x\bigr)\\[0.4em]
	-\,J_{\nu}\!\bigl(\lambda_{\kappa,n}(0,0)\,x\bigr)
\end{pmatrix}\,,
\end{equation}
where $c_{\kappa,n}>0$ is a normalization constant to be specified later.
Using the boundary condition at $x=1\,$, we obtain that the eigenvalues of
$H_\kappa(0)$ are the simple zeros of $J_\nu\,$, as will be detailed in the paragraph below.

\subsection{Symmetries}

In this paragraph we recall the symmetry properties of the unperturbed
spectrum for \(V=0\,\)\,.

\medskip
\noindent
For half--integer orders \(\nu=\kappa+\tfrac12\,\), the entire function
\(z\mapsto \sqrt{z}\,J_\nu(z)\) is an even function when \(\kappa\) is odd,
and an odd function when \(\kappa\) is even.
This parity property immediately yields the following symmetry for the nonzero eigenvalues:
\begin{equation}\label{eq:Symmetry}
	\lambda_{\kappa,-n}(0,0)
	= -\,\lambda_{\kappa,n}(0,0)\,,
	\mbox{ for } n\ge1\,.
\end{equation}

\vspace{0.1cm}\noindent
We recall that the boundary condition leads to the characteristic equation
\(J_\nu(\lambda)=0\)\,. If $\{j_{\nu,n}\}_{n\ge1}$ denotes the sequence of positive zeros of $J_\nu\,$,
then the nonzero eigenvalues of $H_\kappa(0)$ are
\begin{equation}\label{eq:Zeros}
	\lambda_{\kappa,\pm n} (0,0)=\pm j_{\nu,n}\,,\qquad n\ge1\,.
\end{equation}
Moreover, we have seen that $\lambda=0$ is an eigenvalue.
Thus the full spectrum is symmetric and ordered as
\begin{equation}\label{eq:Enumerate}
	\cdots < -\,j_{\nu,2}
	< -\,j_{\nu,1} < 0
	< j_{\nu,1}
	< j_{\nu,2}
	< \cdots\,,
\end{equation}
and we assign $\lambda=0$ to the index $n=0$ in the bi-infinite
enumeration $\{\lambda_{\kappa,n} (0)\}_{n\in\mathbb Z}$\,.

\medskip
\noindent

\subsection{Summary and notation}

The spectrum of the unperturbed operator $H_\kappa(0)$ consists of the simple
eigenvalue $\lambda_{\kappa,0}(0)=0$ and of the nonzero eigenvalues
$\lambda_{\kappa,\pm n}(0)=\pm j_{\nu,n}$ for $n\ge1$, forming a symmetric
bi-infinite sequence indexed by $n\in\mathbb Z$ (see \eqref{eq:Enumerate}).
This enumeration is consistent with the asymptotic formula
\eqref{eq:AsymptAKNS} (see \cite{DLMF},\,10.21 (vi)).

\medskip\noindent
The associated eigenfunctions are given by the regular solutions
$Z_{\kappa,n}^{(0)}\,$. For the zero eigenvalue, one has
\[
Z_{\kappa,0}^{(0)}(x)=\sqrt{2\kappa+1}\binom{x^\kappa}{0}\,,
\]
while for $n\neq0$ they are expressed in terms of Bessel functions.

\medskip\noindent
By symmetry, the eigenfunctions corresponding to $\lambda_{\kappa,n}(0)$ and
$\lambda_{\kappa,-n}(0)$ differ only by a sign in their oscillatory components.
In particular, the normalization constants depend only on $|n|\,$, and
\[
c_{\kappa,-n}=c_{\kappa,n}\,,
\mbox{ for }  n\ge1\,.
\]
Accordingly, we index the spectrum by $n\in\mathbb Z\,$, with $n=0$
corresponding to the zero eigenvalue.

\begin{rem}[Normalization constants] 
\label{rem:kappa-explicit}
We will need the asymptotic behavior of these normalization constants as \( n \to \infty \) (see Section~\ref{subsec:decouple}). In fact, the constants \( c_{\kappa,n} \) can be computed explicitly  using the
following standard finite-interval identity (often referred to as a
\emph{Lommel-type} formula, see (\cite{DLMF}, 10.22.5)): for any $\alpha,\ \mu\in\mathbb R^+$,
\begin{equation}\label{eq:lommel-square}
\int_0^1 x\,J_{\mu}(\alpha x)^2\,dx
=
\frac12\Bigl(J_{\mu}(\alpha)^2 - J_{\mu-1}(\alpha)\,J_{\mu+1}(\alpha)\Bigr)\,.
\end{equation}
Applying \eqref{eq:lommel-square} with $\alpha=j_{\nu,|n|}$ and $\mu=\nu$,
and using $J_\nu(j_{\nu,|n|})=0$ together with the recurrence
$$ J_{\nu-1}(\alpha)+J_{\nu+1}(\alpha)=\frac{2\nu}{\alpha}J_\nu(\alpha)\,,$$ we obtain
$$ J_{\nu-1}(j_{\nu,|n|})=-J_{\nu+1}(j_{\nu,|n|})$$ and hence
\[
\int_0^1 x\,J_{\nu}\!\bigl(j_{\nu,|n|}x\bigr)^2\,dx
=
\frac12\,J_{\nu+1}\!\bigl(j_{\nu,|n|}\bigr)^2\,.
\]
Similarly, applying \eqref{eq:lommel-square} with $\mu=\nu-1$ yields
\[
\int_0^1 x\,J_{\nu-1}\!\bigl(j_{\nu,|n|}x\bigr)^2\,dx
=
\frac12\,J_{\nu-1}\!\bigl(j_{\nu,|n|}\bigr)^2
=
\frac12\,J_{\nu+1}\!\bigl(j_{\nu,|n|}\bigr)^2\,.
\]
Substituting these identities into the normalization condition
$\|Z_{\kappa,n}^{(0)}\|_{L^2(0,1)^2}=1$ (see \eqref{eq:Fundamentalbis}) gives
\begin{equation}
c_{\kappa,n}
=
\frac{1}{\sqrt{\,j_{\nu,|n|}}\,
\bigl|J_{\nu+1}(j_{\nu,|n|})\bigr|}\,,
\mbox{ as }  n\neq0\,.
\end{equation}
The classical asymptotic expansions for Bessel functions and for their
	positive zeros $j_{\nu,n}$ (see, e.g., Watson~\cite{Wa44}) yield
	\begin{equation}\label{asymptconstant}
	c_{\kappa,n}^2
	=
	\frac{\pi}{2}
	+
	\frac{4\nu^{2}-1}{16 \pi\,\bigl(n+\nu/2-1/4\bigr)^{2}}
	+
	O\!\bigl(n^{-4}\bigr)\,,
	\qquad n\to\infty\,.
	\end{equation}
	Finally, in the zero-eigenvalue case, we recall that the normalization condition yields
	\[
	c_{\kappa,0}=\sqrt{2\kappa+1}\,.
	\]
\end{rem}

\bigskip \noindent
We conclude this section with the explicit analysis of the special case
$\kappa=0$. In this case, the computations are particularly simple and allow us to
illustrate the preceding constructions in a fully explicit manner.

\subsection{The case $\kappa=0$}

\noindent
We now consider the case $\kappa = 0$, for which $\nu = \tfrac{1}{2}$ and the singular term $\kappa/x$ disappears from the system. 
The analysis of this particular case is especially interesting (see Section~\ref{kappaegal0}). 
In this setting, the Dirac system \eqref{ODEmatrix} reduces to
\[
Z_1'=\lambda\,Z_2\,,
\qquad 
Z_2'=-\,\lambda\,Z_1\,,
\]
and the regular solution, characterized by the condition $Y_2(0)=0\,$, is given by
\[
Z^{(0)}(x,\lambda)
=
\begin{pmatrix}
\cos(\lambda x)\\[0.2em]
-\sin(\lambda x)
\end{pmatrix}\,.
\]

\noindent
Imposing the boundary condition at $x=1$ yields the characteristic equation
\(
\sin(\lambda)=0\,,
\)
so that the nonzero eigenvalues are explicitly given by
\begin{equation}\label{eq:l0-eigs}
	\lambda_{0,n}(0,0)=n\pi\,,
	\qquad n\in\mathbb Z\setminus\{0\}\,.
\end{equation}

\noindent
For each $n\in\mathbb Z$, the associated eigenfunction is
\begin{equation}\label{eq:l0-eigfct}
	Z_{0,n}^{(0)}(x)
	=
	\begin{pmatrix}
		\cos(n\pi x)\\[0.2em]
		-\sin(n\pi x)
	\end{pmatrix}\,,
\end{equation}
since the $L^2(0,1)^2$-norm of this vector-valued function equals $1\,$.

\begin{rem}
Recall that the Bessel functions of order $\tfrac12$ admit the elementary
representations
\begin{equation}\label{eq:HalfInt}
	J_{1/2}(z)=\sqrt{\frac{2}{\pi z}}\,\sin z\,,
	\qquad
	J_{-1/2}(z)=\sqrt{\frac{2}{\pi z}}\,\cos z\,.
\end{equation}
\end{rem}

\section{Spectral map and the linearized problem at $V=0$.}

In this section we introduce the spectral map and analyze its linearization at
the unperturbed configuration $V=0$. This linearized analysis provides the key
tool for the local inverse results established later.

\subsection{Differential of $\lambda_{\kappa,n}(p,q)$ and the spectral map}

In this subsection we recall the analytic dependence of the eigenvalues on the
AKNS potential $V=(p,q)$ and describe their Fr\'echet differential.

\vspace{0.2cm}\noindent
Following Serier~\cite[Prop.~3.1]{Serier2006}, for each fixed pair
$(\kappa,n)\in\mathbb N\times\mathbb Z$ the map
\[
(p,q)\in L^2(0,1)\times L^2(0,1)\longmapsto
\lambda_{\kappa,n}(p,q)
\]
is real-analytic. Moreover, if $\lambda_{\kappa,n}(p,q)$ is a simple eigenvalue of
$H_\kappa(V)$ with normalized eigenfunction
\[
Z_{\kappa,n}(x;p,q)
=
\begin{pmatrix}
	Z_{1,\kappa,n}(x;p,q)\\
	Z_{2,\kappa,n}(x;p,q)
\end{pmatrix}\,,
\qquad
\|Z_{\kappa,n}(\cdot;p,q)\|_{L^2(0,1)^2}=1\,,
\]
then the Fr\'echet differential at $(p,q)$ in the direction
$(v_1,v_2)\in L^2(0,1)\times L^2(0,1)$ is given by
\begin{equation}\label{eq:DiffGeneral}
	\begin{aligned}
		D_{(p,q)}\lambda_{\kappa,n}(p,q)\cdot (v_1,v_2)
		&=
		\int_0^1
		\Bigl(
		2\,Z_{1,\kappa,n}(x;p,q)\,Z_{2,\kappa,n}(x;p,q)\,v_1(x)\\
		&\qquad
		+\bigl(
		Z_{2,\kappa,n}(x;p,q)^2
		-
		Z_{1,\kappa,n}(x;p,q)^2
		\bigr)\,v_2(x)
		\Bigr)\,dx\,.
	\end{aligned}
\end{equation}

\medskip\noindent
The relation \eqref{eq:DiffGeneral} therefore allows us to compute the differential of the spectral map introduced in \eqref{eq:spectral-map-two-ell}.

\subsection{The linearized problem at $V=0$}

We now investigate the Fr\'echet derivative of the spectral map at the
unperturbed potential \(V=0\), which leads to the formulation of the
linearized inverse problem.
Our objective is to determine whether, for two distinct effective angular momenta \(\kappa_1 \neq \kappa_2\), the associated  map
\[
V=(p,q)\longmapsto
\bigl(\lambda_{\kappa_1,n}(p,q)\,,\,\lambda_{\kappa_2,n}(p,q)\bigr)_{n\in\mathbb Z}
\]
is locally injective at \(V=0\)\,.
Equivalently, the same question can be formulated in terms of the
renormalized eigenvalues \(\widetilde{\lambda}_{\kappa,n}\), since the
renormalization consists in subtracting an explicit function of \(n\)
which is independent of \((p,q)\) and therefore does not affect the Fr\'echet
differential at \(V=0\)\,.

\vspace{0.2cm}\noindent
From the general variation formula~\eqref{eq:DiffGeneral}, one obtains 
\begin{equation}\label{eq:LinEqGeneralBeta}
D_{(p,q)}\lambda_{\kappa,n}(0,0)\cdot (v_1,v_2)= 	\int_0^1
	\Bigl(
	2\,Z_{1,\kappa,n}^{(0)}(x)\,Z_{2,\kappa,n}^{(0)}(x)\,v_1(x)
	+
	\bigl[\bigl(Z_{2,\kappa,n}^{(0)}(x)\bigr)^2
	-
	\bigl(Z_{1,\kappa,n}^{(0)}(x)\bigr)^2\bigr]\,v_2(x)
	\Bigr)\,dx  \,,
\end{equation}
where
\[
Z_{\kappa,n}^{(0)}(x)
=
\begin{pmatrix}
Z_{1,\kappa,n}^{(0)}(x)\\[0.2em]
Z_{2,\kappa,n}^{(0)}(x)
\end{pmatrix}\,,
\]
is the normalized eigenfunction of \(H_\kappa(0)\) associated with the
eigenvalue \(\lambda_{\kappa,n}(0,0)\)\,.

\medskip
\noindent
We begin with the case of nonzero eigenvalues.
Using the results of the previous sections and the symmetry properties of the
unperturbed spectrum, we obtain for every $n\in\mathbb Z^\pm$

\begin{equation}\label{eq:DiffZeroSigned}
\begin{aligned}
D_{(p,q)}\lambda_{\kappa,n}(0,0)\cdot(v_1,v_2)
&=
c_{\kappa,|n|}^{\,2} 
\int_0^1 j_{\nu,|n|}\,x\,
\Bigl(
   \mp\,2\,
   J_{\nu-1}\!\bigl(j_{\nu,|n|}x\bigr)\,
   J_{\nu}\!\bigl(j_{\nu,|n|}x\bigr)\,v_1(x)
\\[-0.2em]
&\qquad\qquad\qquad
   +\,
   \bigl[
      J_{\nu}\!\bigl(j_{\nu,|n|}x\bigr)^2
      -
      J_{\nu-1}\!\bigl(j_{\nu,|n|}x\bigr)^2
   \bigr]\,v_2(x)
\Bigr)\,dx \,.
\end{aligned}
\end{equation}

\medskip
\noindent
We next consider the zero eigenvalue.
Recalling that
\[
Z_{\kappa,0}^{(0)}(x)
= c_{\kappa,0} 
\begin{pmatrix}
x^\kappa\\[0.2em]
0
\end{pmatrix}\,,
\qquad
c_{\kappa,0}=\sqrt{2\kappa+1}\,,
\]
substituting this expression into~\eqref{eq:DiffGeneral} yields
\begin{equation}\label{aazz}
D_{(p,q)}\lambda_{\kappa,0}(0,0)\cdot (v_1,v_2)
=
-\,(2\kappa+1) \int_0^1 x^{2\kappa}\,v_2(x)\,dx\, .
\end{equation}


\vspace{0.2cm}\noindent
The structure of system \eqref{eq:DiffZeroSigned} suggests that the contributions of
$v_1$ and $v_2$ can be separated. We now formalize this observation by
introducing a bounded isomorphism on the target space which exactly decouples
the differential of the spectral map.

\subsection{Decoupling of the differential via a continuous isomorphism}
\label{subsec:decouple}

We show that, up to a bounded isomorphism on the target space, the differential
of the spectral map can be reduced to a fully decoupled system. This allows us
to study independently the contributions of $v_1$ and $v_2$. We denote this
differential by
\[
S := D_{(p,q)}\mathcal S_{\kappa_1,\kappa_2}(0,0)\,.
\]

\medskip\noindent
\paragraph{Notation.}
For each $\kappa\in\{\kappa_1,\kappa_2\}$ and each $n\ge1$ (with $\nu=\kappa+\tfrac12$), 
we introduce the bounded linear functionals on $L^2(0,1)$
\[
A_{\kappa,n}(v_1)
:=
\int_0^1
2 j_{\nu,n} x\,J_{\nu-1}\!\bigl(j_{\nu,n}x\bigr)\,
J_{\nu}\!\bigl(j_{\nu,n}x\bigr)\,v_1(x)\,dx\,,
\]
\[
B_{\kappa,n}(v_2)
:=
\int_0^1
j_{\nu,n} x
\Bigl(
J_{\nu}\!\bigl(j_{\nu,n}x\bigr)^2
-
J_{\nu-1}\!\bigl(j_{\nu,n}x\bigr)^2
\Bigr)\,v_2(x)\,dx\,.
\]
In this notation, for $n\ge1\,$,
\[
D_{(p,q)}\lambda_{\kappa,\pm n}(0,0)\cdot(v_1,v_2)
=
c_{\kappa,n}^2\bigl(\mp A_{\kappa,n}(v_1)+B_{\kappa,n}(v_2)\bigr)\,,
\]
while for the zero mode one has
\begin{equation}\label{eq:B-diff-zero}
	D_{(p,q)}\lambda_{\kappa,0}(0,0)\cdot (v_1,v_2)
	=
	-\,c_{\kappa,0}^2 \int_0^1 x^{2\kappa}\,v_2(x)\,dx\,,
	\qquad
	c_{\kappa,0}=\sqrt{2\kappa+1}\,.
\end{equation}

\medskip\noindent
For $\kappa\in\{\kappa_1,\kappa_2\}\,$, define
\[
d_\kappa(v):=\bigl(d_{\kappa,n}(v)\bigr)_{n\in\mathbb Z}
:=
\bigl(D_{(p,q)}\lambda_{\kappa,n}(0,0)\cdot v\bigr)_{n\in\mathbb Z}
\in \ell^2(\mathbb Z)\,,
\qquad v=(v_1,v_2)\,,
\]
so that
\[
S(v)=(d_{\kappa_1}(v),\,d_{\kappa_2}(v))\in \ell^2(\mathbb Z)\times\ell^2(\mathbb Z)\,.
\]

\medskip\noindent
Now, for fixed $\kappa\,$, define the linear map
\[
\mathcal U_\kappa:\ell^2(\mathbb Z)\longrightarrow
\mathbb R\times \ell^2(\mathbb N^*)\times \ell^2(\mathbb N^*)
\]
by
\begin{equation}\label{eq:Uell-def}
	\mathcal U_\kappa(a)
	=
	\left(
	\frac{a_0}{c_{\kappa,0}^{\,2}}\,,\;
	\Bigl(\frac{a_{-n}-a_n}{2\,c_{\kappa,n}^{\,2}}\Bigr)_{n\ge1}\,,\;
	\Bigl(\frac{a_{-n}+a_n}{2\,c_{\kappa,n}^{\,2}}\Bigr)_{n\ge1}
	\right)\,,
\end{equation}
where $\mathbb{N}^* := \{1,2,3,\dots\}$ denotes the set of positive integers.
Using \eqref{asymptconstant}, we see that $(c_{\kappa,n})_{n\in\mathbb Z}$ is uniformly bounded above and below, so
$\mathcal U_\kappa$ is a bounded isomorphism. 
Applying $\mathcal U_\kappa$ to $d_\kappa(v)$ and using the formulas above yields the
exact decoupling
\begin{equation}\label{eq:Uell-T}
	\mathcal U_\kappa\bigl(d_\kappa(v)\bigr)
	=
	\left(
	-\int_0^1 x^{2\kappa}v_2\,,\;
	\bigl(A_{\kappa,n}(v_1)\bigr)_{n\ge1}\,,\;
	\bigl(B_{\kappa,n}(v_2)\bigr)_{n\ge1}
	\right)\,.
\end{equation}

\medskip\noindent
Finally, set $\mathcal U:=\mathcal U_{\kappa_1}\times \mathcal U_{\kappa_2}$, which
is a bounded isomorphism on $\ell^2(\mathbb Z)\times\ell^2(\mathbb Z)$\,. Then
\begin{equation}\label{eq:U-T-block}
	(\mathcal U\circ S)(v_1,v_2)
	=
	\Bigl(
	\mathcal M(v_2),\;
	\mathcal A_{\kappa_1,\kappa_2}(v_1),\;
	\mathcal B_{\kappa_1,\kappa_2}(v_2)
	\Bigr)\,,
\end{equation}
where
\[
\mathcal M(v_2)
=
\left(
-\int_0^1 x^{2\kappa_1}v_2\,,\;
-\int_0^1 x^{2\kappa_2}v_2
\right)\in\mathbb R^2\,\,,
\]
and 
\[
\mathcal A_{\kappa_1,\kappa_2}(v_1)
=
\Bigl(
(A_{\kappa_1,n}(v_1))_{n\ge1}\,,\;
(A_{\kappa_2,n}(v_1))_{n\ge1}
\Bigr)\,,
\qquad
\mathcal B_{\kappa_1,\kappa_2}(v_2)
=
\Bigl(
(B_{\kappa_1,n}(v_2))_{n\ge1}\,,\;
(B_{\kappa_2,n}(v_2))_{n\ge1}
\Bigr)\,.
\]
In particular, $\mathcal U\circ S$ is block diagonal: $v_1$ only enters
$\mathcal A_{\kappa_1,\kappa_2}$, whereas $v_2$ only enters
$(\mathcal M,\mathcal B_{\kappa_1,\kappa_2})\,$.

\subsection{Reformulation of the injectivity problem}\label{kappaegal0}

We now reformulate the injectivity of the Fr\'echet differential of the spectral
map $\mathcal S_{\kappa_1,\kappa_2}$ at $(p,q)=(0,0)\,$. By definition, injectivity
amounts to characterizing all perturbations $(v_1,v_2)$ such that
\[
S(v_1,v_2)=0\,,
\qquad
S=D_{(p,q)}\mathcal S_{\kappa_1,\kappa_2}(0,0)\,.
\]

\medskip\noindent
Since $\mathcal U$ is a bounded isomorphism on the target space, this condition
is equivalent to
\[
(\mathcal U\circ S)(v_1,v_2)=0\,.
\]
Using the block diagonal structure \eqref{eq:U-T-block}, the kernel condition
reduces to the decoupled system
\begin{equation}\label{eq:kernel-decoupled}
	\mathcal A_{\kappa_1,\kappa_2}(v_1)=0\,,
	\qquad
	\mathcal M(v_2)=0\,,
	\qquad
	\mathcal B_{\kappa_1,\kappa_2}(v_2)=0\,.
\end{equation}
This decoupling also reflects the choice of boundary conditions, which is encoded in the structure of the eigenfunctions. Thus, the study of the kernel separates into two independent problems: one
involving only the component $v_1$, governed by
$\mathcal A_{\kappa_1,\kappa_2}$, and one involving only the component $v_2$,
governed by $(\mathcal M,\mathcal B_{\kappa_1,\kappa_2})\,$.

\medskip\noindent
For a fixed effective angular momentum $\kappa$, the above conditions reduce to
\begin{equation}\label{eq:single-ell-kernel}
	A_{\kappa,n}(v_1)=0,\quad n\ge1\,,
	\qquad
	B_{\kappa,n}(v_2)=0,\quad n\ge1\,,
	\qquad
	\int_0^1 x^{2\kappa}v_2(x)\,dx=0\,.
\end{equation}

\medskip\noindent
As an illustration, consider the simple case $\kappa=0\,$. Using the explicit
expressions of the Bessel functions of order $\pm\tfrac12\,$, one obtains
\[
x\,J_{-1/2}(n\pi x)\,J_{1/2}(n\pi x)
= \tfrac12 \sin(2n\pi x)\,,
\qquad
x\bigl(J_{1/2}(n\pi x)^2 - J_{-1/2}(n\pi x)^2\bigr)
= -\cos(2n\pi x)\,.
\]
Hence \eqref{eq:single-ell-kernel} becomes
\[
\int_0^1 \sin(2n\pi x)\,v_1(x)\,dx=0\,,
\qquad n\ge1,
\]
and
\[
\int_0^1 \cos(2n\pi x)\,v_2(x)\,dx=0\,,
\qquad n\ge1,
\qquad
\int_0^1 v_2(x)\,dx=0\,.
\]
These relations show that $v_1$ is even and $v_2$ is odd with respect to
$x=\tfrac12\,$. Conversely, if $v_1$ is even and $v_2$ is odd with respect
to $x=\tfrac12\,$, then all the above integrals vanish. Hence these parity
conditions are necessary and sufficient for $S(v_1,v_2)=0$ in the case
$\kappa=0$.


\section{Kneser--Sommerfeld--Type Expansions}\label{s3}

The classical Kneser--Sommerfeld identity provides a series expansion over the
zeros \( j_{\nu,n} \) of \( J_\nu \)\,.
Its correct form, first given by Buchholz~\cite{Buch47a} and later clarified
by Hayashi~\cite{Hay82} and Martin~\cite{Martin21}, differs from the formula
stated by Watson, which omits an essential integral term.
The valid expansion
\eqref{eq:KS-correct} played a central role in our previous analysis of the
radial Schr\"odinger operator:
\begin{equation} \label{eq:KS-correct}
\sum_{n\ge1}
\frac{
   J_\nu(j_{\nu,n}x)\, J_\nu(j_{\nu,n}X)
}{
   (z^{2} - j_{\nu,n}^{2})\,[J'_\nu(j_{\nu,n})]^2
}
=
\frac{\pi}{4\,J_\nu(z)}\,J_\nu(xz)\,
\bigl[
   J_\nu(z)\,Y_\nu(Xz)
   - Y_\nu(z)\,J_\nu(Xz)
\bigr]\,,
\end{equation}
for $0<x \leq X \leq 1$\,.

\medskip
\noindent
In the present AKNS setting, the linearized system~\eqref{eq:DiffZeroSigned} involves both squared and 
mixed Bessel products.  
The relevant combinations are, with $\nu=\kappa+\tfrac12$\,, 
\[
J_{\nu}(j_{\nu,n}x)^{2}\,,\quad
J_{\nu-1}(j_{\nu,n}x)^{2}\quad \mbox{ and }\quad
J_{\nu-1}(j_{\nu,n}x)\,J_{\nu}(j_{\nu,n}x)\,.
\]

\vspace{0.2cm}\noindent
However, the last two expressions fall outside the framework of the classical
Kneser--Sommerfeld expansion, which treats only diagonal products of the form
\(J_\nu(x j_{\nu,n})J_\nu(X j_{\nu,n})\)\,.

\medskip\noindent
To handle the AKNS structure, we therefore require 
\emph{modified} Kneser--Sommerfeld--type identities.
In what follows, we now state three additional identities of the same type. For simplicity, we assume that $\nu\in\mathbb{R^+}\backslash\mathbb{N}\,$, although the formulas extend to arbitrary complex values of $\nu\,$.\footnote{Here we use the definition of $Y_n(z)$ recalled above.}

\begin{prop}\label{KSgen}
Let \(\nu\in\mathbb{R}_+\setminus\mathbb{N}\), \(z\neq 0\),  \(z\neq j_{\nu,n}\) and \(0<x\le X\le 1\)\,.
The following Kneser--Sommerfeld--type identities hold:
\begin{equation}\label{eq:KS-nu-one}
\begin{aligned}
\sum_{n\ge1}
\frac{
   J_{\nu-1}(j_{\nu,n}x)\, J_{\nu-1}( j_{\nu,n}X)
}{
   (z^{2}-j_{\nu,n}^{2})\,\bigl[J'_\nu(j_{\nu,n})\bigr]^{2}
}
&=
-\frac{\nu }{z^2}\, (xX)^{\nu-1}
\\[0.3em]
&\quad+
\frac{\pi}{4J_\nu(z)}\,J_{\nu-1}(xz)\,
\bigl(
   J_\nu(z)\,Y_{\nu-1}(Xz)
   - Y_\nu(z)\,J_{\nu-1}(Xz)
\bigr)\,,
\end{aligned}
\end{equation}

\begin{equation}\label{eq:KS-nu-minus-onebis}
\begin{aligned}
\sum_{n\ge1}
\frac{
   J_{\nu-1}(j_{\nu,n}x)\,J_{\nu}(j_{\nu,n}X)
}{
   (z^{2}-j_{\nu,n}^{2})\,j_{\nu,n}\,\bigl[J'_\nu(j_{\nu,n})\bigr]^{2}
}
&=
\frac{1}{2z^2}\, x^{\nu-1} (X^{-\nu} - X^\nu)
\\[0.3em]
&\quad+
\frac{\pi}{4z\,J_{\nu}(z)}\,J_{\nu-1}(xz)\,
\bigl(
   J_{\nu}(z)\,Y_{\nu}(Xz)
   - Y_{\nu}(z)\,J_{\nu}(Xz)
\bigr)\,,
\end{aligned}
\end{equation}

\begin{equation}\label{eq:KS-nu-minus-onesym}
\begin{aligned}
\sum_{n\ge1}
\frac{
   J_{\nu-1}(j_{\nu,n}X)\,J_{\nu}(j_{\nu,n}x)
}{
   (z^{2}-j_{\nu,n}^{2})\,j_{\nu,n}\,\bigl[J'_\nu(j_{\nu,n})\bigr]^{2}
}
&=
-\frac{1}{2z^2}\, x^\nu X^{\nu-1}
\\[0.3em]
&\quad+
\frac{\pi}{4z\,J_{\nu}(z)}\,J_{\nu}(xz)\,
\bigl(
   J_{\nu}(z)\,Y_{\nu-1}(Xz)
   - Y_{\nu}(z)\,J_{\nu-1}(Xz)
\bigr)\,.
\end{aligned}
\end{equation}
\end{prop}

\begin{proof}
We follow, step by step, the contour-integral argument of Watson for the
Kneser-Sommerfeld formula (see \cite{Wa44}), in the corrected form later
clarified by Buchholz, Hayashi and Martin. We first give the proof of \eqref{eq:KS-nu-one}.

\vspace{0.2cm}\noindent
Let $z\in\mathbb{C}\,$ with, for instance, $\Re z>0$ and $z\neq j_{\nu,n}$ for all $n\ge1\,$.
Following Watson's approach, we introduce, for $w\in\mathbb{C}$ and fixed
$0<X\le 1$, the auxiliary function
\begin{equation}\label{eq:defW}
W_\nu(w,X)
:= J_\nu(w)\,Y_{\nu-1}(Xw)\;-\;Y_\nu(w)\,J_{\nu-1}(Xw)\,.
\end{equation}
Using the small-$w$ asymptotics of $J_\nu$ and $Y_\nu$ (see
\eqref{singularite}), one readily obtains
\begin{equation}\label{eq:W_small_w}
W_\nu(w,X)
= \frac{2 X^{\nu-1}}{\pi w} \;+\; O(1)\,,
\qquad w\to 0\,.
\end{equation}
Recall the Hankel functions
\begin{equation}\label{Hankel}
H^{(1)}_\nu(w)=J_\nu(w)+i\,Y_\nu(w)\,, \qquad
H^{(2)}_\nu(w)=J_\nu(w)-i\,Y_\nu(w)\,.
\end{equation}
A short computation gives the equivalent representation
\begin{equation}\label{eq:W-Hankel}
W_\nu(w,X)
= -\,\frac{1}{2i}
\Bigl[
   H^{(1)}_\nu(w)\,H^{(2)}_{\nu-1}(Xw)
   - H^{(2)}_\nu(w)\,H^{(1)}_{\nu-1}(Xw)
\Bigr]\,.
\end{equation}
Finally, from the large-$w$  asymptotics (\cite{Lebedev72},
(5.11.4)-(5.11.5)), valid for $|\arg w|\le\pi-\delta$\,, $\delta \in (0,\pi)\,$,
\begin{equation}\label{AsymptH}
H^{(1)}_\nu(w)\sim 
\sqrt{\frac{2}{\pi w}}\,
e^{i(w-\frac{\nu\pi}{2}-\frac{\pi}{4})}\,,
\qquad
H^{(2)}_\nu(w)\sim 
\sqrt{\frac{2}{\pi w}}\,
e^{-i(w-\frac{\nu\pi}{2}-\frac{\pi}{4})}\,,
\end{equation}
we obtain, for $|w|$ large with $\Re w>0\,$,
\begin{equation}\label{eq:W_large_w}
|W_\nu(w,X)| \lesssim \frac{e^{(1-X)|\Im w|}}{|w|}\,,\qquad 0<X\le1\,.
\end{equation}

\vspace{0.2cm}\noindent
We consider the contour integral
\begin{equation}\label{eq:IBMRbis}
I_{B,M, \epsilon} = \oint_{C_{B,M,\epsilon}} \frac{W_\nu(w,X)}{w^{2}-z^{2}} \,\frac{w\,J_{\nu-1}(xw)}{J_\nu(w)}\,dw\,,
\end{equation}
where \( C_{B,M,\epsilon} \)  is the  rectangle in the half-plane $\Re w \geq 0\,$,
with vertices
\[
\pm iB, \qquad \pm iB + M\pi + \frac{(2\nu+1)\pi}{4}\,,
\]
for $B>0$ and $M\in\mathbb{N}$ large enough, indented at the origin with a half-circle of radius \(\epsilon > 0\) in the half-plane \(\Re w > 0\)\,: 


\begin{center}
\begin{tikzpicture}[scale=1.05, line cap=round, line join=round]
  \def\B{2.2}        
  \def\Len{4.8}      
  \def\eps{0.45}     

  \draw[->] (-0.8,0) -- (\Len+1.0,0) node[below] {$\Re w$};
  \draw[->] (0,-\B-0.8) -- (0,\B+0.8) node[left] {$\Im w$};

  \coordinate (Top) at (0,\B);
  \coordinate (Bot) at (0,-\B);
  \coordinate (TopCut) at (0,\eps);
  \coordinate (BotCut) at (0,-\eps);

  \coordinate (TopR) at (\Len,\B);
  \coordinate (BotR) at (\Len,-\B);

  \draw[thick,->] (BotCut) -- (Bot);
  \draw[thick,->] (Bot) -- (BotR);
  \draw[thick,->] (BotR) -- (TopR);
  \draw[thick,->] (TopR) -- (Top);
  \draw[thick,->] (Top) -- (TopCut);

  \draw[thick,->] (TopCut)
    arc[start angle=90, end angle=-90, radius=\eps];

  \fill (Top) circle (1.1pt) node[left] {$iB$};
  \fill (Bot) circle (1.1pt) node[left] {$-iB$};
  \fill (TopR) circle (1.1pt)
    node[right] {$iB+M\pi+\frac{(2\nu+1)\pi}{4}$};
  \fill (BotR) circle (1.1pt)
    node[right] {$-iB+M\pi+\frac{(2\nu+1)\pi}{4}$};

  \fill (TopCut) circle (1.0pt);
  \fill (BotCut) circle (1.0pt);
  \node[left] at (0,\eps) {$i\epsilon$};
  \node[left] at (0,-\eps) {$-i\epsilon$};
  \node at (1.2*\eps,0.1) {\small $\epsilon$};

  \fill (0,0) circle (1.2pt) node[below left] {$0$};

  \coordinate (z) at (1.8,0.6);
  \fill (z) circle (1.4pt) node[above right] {$z$};

  \node at (\Len*0.6,\B*0.8) {$C_{B,M,\epsilon}$};
\end{tikzpicture}
\end{center}


\medskip\noindent
Using  again the small-$w$ asymptotics of $J_\nu$ and $Y_\nu$ (see
\eqref{singularite}), one readily obtains 
\begin{equation}\label{eq:W_small_wbis}
\frac{W_\nu(w,X)}{w^{2}-z^{2}}
\,\frac{w\,J_{\nu-1}(xw)}{J_\nu(w)} = -\frac{4\nu}{\pi z^2 w} (xX)^{\nu-1} + O(1)\,,
\quad w\to 0\,.
\end{equation}
Using the parity identities 
\[
H^{(1)}_{\nu}(-w) = e^{-i\pi\nu}H^{(1)}_{\nu}(w)\,, \quad
H^{(2)}_{\nu}(-w) = e^{i\pi\nu}H^{(2)}_{\nu}(w)\,,
\]
which hold for all \( w \) on the imaginary axis, we observe that \( W_\nu(w,X) \) is an odd function of \( w \) on this axis.
Furthermore, the map
\[
w \mapsto \frac{w\,J_{\nu-1}(xw)}{J_\nu(w)}
\]
is even in \( w \)\,. Thus, the contribution from the vertical union of the intervals
\([-iB,-i\epsilon]\cup[i\epsilon,iB]\) cancels out.

\medskip
\noindent
Let us set
\begin{equation}\label{eq:def-fw}
f(w)
=
\frac{W_\nu(w,X)}{w^{2}-z^{2}}\,
\frac{w\,J_{\nu-1}(xw)}{J_\nu(w)}\,.
\end{equation}
Using \eqref{eq:W_small_wbis},  the contribution of the integral over the small circle
centered at the origin and of radius \(\epsilon\) (traversed in the
counterclockwise direction) converges, in the limit \(\epsilon\to 0\)\,, to
\begin{equation}\label{eq:small-circle-contribution}
\frac{4i\nu}{z^{2}}\,(xX)^{\nu-1}\,.
\end{equation}

\vspace{0.2cm}\noindent
Using Bessel asymptotics (\cite{Lebedev72}, (5.11.6))
\begin{equation}\label{AsymptJ}
J_\nu(w)
=\sqrt{\frac{2}{\pi w}}
\left[
  \cos\!\Bigl(w-\tfrac{\nu\pi}{2}-\tfrac{\pi}{4}\Bigr)
  + O\Bigl(\frac{e^{|\Im w|}}{|w|}\Bigr)
\right]\,,
\quad |\arg w|\le\pi-\delta \,,\ |w| \to +\infty\,,
\end{equation}
with $\delta \in (0,\pi)$,
we deduce that no Bessel zero $j_{\nu,n}$ lies on the vertical segments and that on the three other sides,
\begin{equation}\label{eq:integrand_bound}
\left|
\frac{
   H^{(1)}_\nu(w)\,H^{(2)}_{\nu-1}(Xw)
   - H^{(2)}_\nu(w)\,H^{(1)}_{\nu-1}(Xw)
}{w^{2}-z^{2}}
\frac{w\,J_{\nu-1}(xw)}{J_\nu(w)}
\right|
\;\lesssim\;
\frac{e^{(x-X)|\Im w|}}{|w|^{2}}\,.
\end{equation}
Since by the assumption in the proposition  $x\le X$, the integrand decays like $|w|^{-2}$ on these sides.
Letting \(B,M\to\infty\), all these contributions vanish. Therefore, by the
residue theorem, letting \(B,M\to\infty\) in \eqref{eq:IBMRbis} and
\(\epsilon\to 0\)\,, we obtain
\begin{equation}\label{thmresidus}
\frac{4i\nu}{z^{2}}\,(xX)^{\nu-1}
=
2\pi i
\sum \operatorname*{Res}(f)\,,
\end{equation}
where the sum runs over the poles inside the contour, namely
\[
w=z
\quad\text{and}\quad
w=j_{\nu,n}, \qquad n\ge1\,.
\]

\vspace{0.2cm}\noindent
a. Residue at $w=j_{\nu,n}$ : the integrand has a simple pole, giving
\begin{equation}\label{eq:res_jnun}
\Res_{w=j_{\nu,n}}(f)
=
\frac{
   J_{\nu-1}(X j_{\nu,n})\, j_{\nu,n}\, J_{\nu-1}(x j_{\nu,n})\,
   Y_\nu(j_{\nu,n})
}{
   (z^{2}-j_{\nu,n}^{2})\,J'_\nu(j_{\nu,n})
}\,.
\end{equation}
We use the identity
\[
Y_\nu(j_{\nu,n})
=
\frac{
   Y_\nu(j_{\nu,n})J'_\nu(j_{\nu,n})
   - J_\nu(j_{\nu,n})Y'_\nu(j_{\nu,n})
}{J'_\nu(j_{\nu,n})}\,,
\]
and the Wronskian formula (\cite{Lebedev72}, (5.9.2))
\[
J_\nu(z)Y'_\nu(z)-J'_\nu(z)Y_\nu(z)=\frac{2}{\pi z}\,,
\]
to obtain
\[
Y_\nu(j_{\nu,n})
=
-\frac{2}{\pi\, j_{\nu,n}\, J'_\nu(j_{\nu,n})}\,.
\]
Substituting this into \eqref{eq:res_jnun} gives
\begin{equation}\label{eq:res_final_jnun}
\Res_{w=j_{\nu,n}}(f)=
-\frac{2}{\pi}\,
\frac{
   J_{\nu-1}(X j_{\nu,n})\, J_{\nu-1}(x j_{\nu,n})
}{
   (z^{2}-j_{\nu,n}^{2})\,[J'_\nu(j_{\nu,n})]^2
}\,.
\end{equation}

\vspace{0.2cm}\noindent
b. Residue at $w=z$ : a direct computation shows:
\begin{equation}\label{eq:res_z}
\Res_{w=z}(f)
=
\frac{1}{2J_\nu(z)}\,J_{\nu-1}(xz)\,
\bigl[
   J_\nu(z)\,Y_{\nu-1}(Xz)
   - Y_\nu(z)\,J_{\nu-1}(Xz)
\bigr]\,.
\end{equation}
Summing the residues \eqref{eq:res_final_jnun} and \eqref{eq:res_z} and invoking
\eqref{thmresidus} yields precisely the identity \eqref{eq:KS-nu-one}.

\vspace{0.5cm}\noindent 
To prove the identity \eqref{eq:KS-nu-minus-onebis}, we introduce a new auxiliary  function
\begin{equation}\label{eq:defWmodifiedbis}
\tilde W_\nu(w,X) := J_\nu(w)\,Y_{\nu}(Xw) - Y_\nu(w)\,J_{\nu}(Xw)\,,
\end{equation}
and we consider the contour integral
\begin{equation}\label{eq:IBMRter}
I_{B,M,\epsilon} = \oint_{C_{B,M,\epsilon}} \frac{\tilde W_\nu(w,X)}{w^{2}-z^{2}} \,\frac{J_{\nu-1}(xw)}{J_\nu(w)}\,dw\,,
\end{equation}
where \( C_{B,M,\epsilon} \) is the same rectangle as above, indented at the origin with a half-circle of radius \( \epsilon > 0 \)\,. We conclude in exactly the same way.

\vspace{0.5cm}\noindent
Similarly, to prove the identity \eqref{eq:KS-nu-minus-onesym}, we use again the auxiliary function \eqref{eq:defW}
and consider the contour integral
\begin{equation}\label{eq:IBMRtersym}
I_{B,M,\epsilon} = \oint_{C_{B,M,\epsilon}} \frac{W_\nu(w,X)}{w^{2}-z^{2}} \,\frac{J_{\nu}(xw)}{J_\nu(w)}\,dw\,,
\end{equation}
where \( C_{B,M,\epsilon} \) is the same contour as before. The proof then follows analogously.
\end{proof}

\vspace{0.2cm}\noindent
The following result is readily obtained:

\begin{coro}\label{KSsomme} 
Let $\nu\in\mathbb{R^+}\backslash\mathbb{N}$, and  \( x \in (0,1] \)\,. Then, for all $z \notin \{ 0, j_{\nu,n}\}_{n \geq 1}$,
\begin{equation}\label{eq:KS-nu-minus-oneadd}
\begin{aligned}
\sum_{n \geq 1}
\frac{ x \, J_{\nu-1}(j_{\nu,n}x)\,J_{\nu}(j_{\nu,n}x)}{(z^{2}-j_{\nu,n}^{2})\,j_{\nu,n}\,\bigl[J'_\nu(j_{\nu,n})\bigr]^{2}}
&=
\frac{1}{4z^2 }\, (1-2x^{2\nu})
\\[0.3em]
&\quad+
\frac{\pi \, x}{8z\,J_{\nu}(z)}\, \left[
   J_{\nu}(z)\,\bigl(J_{\nu-1}(zx)\,Y_{\nu}(zx) + Y_{\nu-1}(zx)\,J_{\nu}(zx)\bigr)
   \right.
\\[0.3em]
&\qquad
\left.
   - 2 Y_{\nu}(z)\,J_{\nu-1}(zx)\,J_{\nu}(zx)
\right]\,.
\end{aligned}
\end{equation}
\end{coro}

\begin{proof}
The result follows by adding \eqref{eq:KS-nu-minus-onebis} and \eqref{eq:KS-nu-minus-onesym} with \( x = X \)\, and multiplying by $\frac{x}{2}$.
\end{proof}

\noindent


\section{Application of the Kneser--Sommerfeld representation}
\label{s4}

\subsection{Preliminaries}

We have seen that for a fixed  effective angular momentum $\kappa$ the linearized condition
\[
D_{(p,q)}\lambda_{\kappa,n} (0,0) \cdot (v_1, v_2) = 0\,,
\qquad \text{for all } n \in \mathbb{Z}\,,
\]
is equivalent to the relations (see \eqref{eq:single-ell-kernel})
\begin{equation}\label{eq:v1-system}
	\int_0^1 x \, J_{\nu-1}(j_{\nu,n}x) \, J_{\nu}(j_{\nu,n}x) \, v_1(x) \, dx = 0\,,
	\qquad n \geq 1\,,
\end{equation}
and 
\begin{equation}\label{eq:v2-system}
	\int_0^1 x \left[ J_{\nu}(j_{\nu,n}x)^2 - J_{\nu-1}(j_{\nu,n}x)^2 \right] v_2(x) \, dx = 0\,,
	\qquad n \geq 1\,,
\end{equation}
with the additional constraint
\begin{equation}\label{eq:zero-eigenvalue-constraintbis}
	\int_0^1 x^{2\kappa} \, v_2(x) \, dx = 0\,.
\end{equation}
Let us first examine the conditions (\ref{eq:v2-system})-\eqref{eq:zero-eigenvalue-constraintbis}. We use the classical Kneser-Sommerfeld expansion (\ref{eq:KS-correct}) and the relation (\ref {eq:KS-nu-one}) with \(x = X\)\,. We multiply (\ref{eq:v2-system}) by
$$
\frac{1}{(z^{2}-j_{\nu,n}^{2})\,\bigl[J'_\nu(j_{\nu,n})\bigr]^{2}}
$$
with $z \notin \{0, j_{\nu,n}\}_{n \geq 1}$. Summing over \(n\), and using  $2\nu-1=2\kappa$,  we obtain for such $z$,
\begin{multline}\label{eq:final-relation}
	- \frac{\nu}{z^2} \int_0^1 x^{2\nu-1} v_2(x) \, dx
	+ \frac{\pi}{4 J_\nu(z)} \int_0^1 x \Bigl( J_\nu(zx) \left[ J_\nu(z) Y_\nu(zx) - Y_\nu(z) J_\nu(zx) \right] \\
	- J_{\nu-1}(zx) \left[ J_\nu(z) Y_{\nu-1}(zx) - Y_\nu(z) J_{\nu-1}(zx) \right] \Bigr) v_2(x) \, dx = 0\,.
\end{multline}
Since $2\nu-1=2\kappa$, the first integral in
\eqref{eq:final-relation} vanishes thanks to the constraint
\eqref{eq:zero-eigenvalue-constraintbis}.  
Hence we obtain the simplified identity:
\begin{multline}\label{eq:final-relationfinal}
	\int_0^1 x \Bigl( J_\nu(zx) \left[ J_\nu(z) Y_\nu(zx) - Y_\nu(z) J_\nu(zx) \right] \\
	- J_{\nu-1}(zx) \left[ J_\nu(z) Y_{\nu-1}(zx) - Y_\nu(z) J_{\nu-1}(zx) \right] \Bigr) v_2(x) \, dx = 0\,,
\end{multline}
which can be rewritten for $z \notin \{0, j_{\nu,n}\}_{n \geq 1}$ as
\begin{equation}\label{eq:decomposed-relation}
	\begin{split}
		& J_\nu(z) \int_0^1 x \Bigl( J_\nu(zx) Y_\nu(zx) - J_{\nu-1}(zx) Y_{\nu-1}(zx) \Bigr) v_2(x) \, dx \\
		&\quad - Y_\nu(z) \int_0^1 x \Bigl( J_\nu(zx)^2 - J_{\nu-1}(zx)^2 \Bigr) v_2(x) \, dx = 0\,.
	\end{split}
\end{equation}

\medskip\noindent
By continuity with respect to $z$, the identity (\ref{eq:decomposed-relation}) extends to $z \in \C^*\,$.

\vspace{0.5cm}
\noindent
In the same way, using Corollary~\ref{KSsomme}, we obtain from \eqref{eq:v1-system}, for $z \notin \{0, j_{\nu,n}\}_{n \geq 1}$ 
\begin{equation}\label{eqv1}
\begin{split}
    \int_0^1 (1 - 2x^{2\nu})\, v_1(x) \, dx
    &+ \frac{\pi}{2 J_\nu(z)} \int_0^1 (zx)\,
    \Bigl(
        J_{\nu}(z)\bigl[
            J_{\nu-1}(zx) Y_\nu(zx) + Y_{\nu-1}(zx) J_\nu(zx)
        \bigr] \\
    &\qquad\qquad\qquad
        -2Y_\nu(z) J_{\nu-1}(zx) J_{\nu}(zx)
    \Bigr) v_1(x) \, dx
    = 0 \, .
\end{split}
\end{equation}
Thus, using the large-$w$ asymptotics for Bessel functions \eqref{AsymptJ} together with the asymptotics for $Y_\nu(w)$ (see \cite{Lebedev72}, (5.11.7)),
\begin{equation}\label{AsymptY}
Y_\nu(w)
=
\sqrt{\frac{2}{\pi w}}
\left[
\sin\!\Bigl(w-\frac{\nu\pi}{2}-\frac{\pi}{4}\Bigr)
+O\!\left(\frac{e^{|\Im w|}}{|w|}\right)
\right]\,,
\qquad |\arg w|\le \pi-\delta\,,\quad |w|\to+\infty \,,
\end{equation}
we deduce using the Riemann-Lebesgue lemma,  that the integral in \eqref{eqv1} is $o(1)$ as $z\to +\infty$ away from the points $j_{\nu,n}\,$. Consequently,
\begin{equation}\label{annulation}
    \int_0^1 \bigl(1 - 2x^{2\nu}\bigr)\, v_1(x)\, dx = 0 \,.
\end{equation}
As a consequence, for all  $z \notin \{0, j_{\nu,n}\}_{n \geq 1}$, we obtain
\begin{equation}\label{decouiplagev1}
\begin{aligned}
&J_{\nu}(z)\int_0^1 x \bigl(J_{\nu-1}(zx)Y_{\nu}(zx) + Y_{\nu-1}(zx)J_{\nu}(zx)\bigr) v_1(x) \, dx \\
&\quad + Y_{\nu}(z)\int_0^1 x \bigl(-2J_{\nu-1}(zx)J_{\nu}(zx)\bigr) v_1(x) \, dx = 0\,,
\end{aligned}
\end{equation}
and this identity extends to all $z\in\mathbb{C}^*\,$.

\vspace{0.2cm}\noindent
Now, we use the vector functions introduced by Serier:
\[
\Phi_\kappa(x)
=
\begin{pmatrix}
	\Phi_{\kappa,1}(x)\\[0.3em]
	\Phi_{\kappa,2}(x)
\end{pmatrix}
=
\frac{\pi x}{2}
\begin{pmatrix}
	-2\,J_{\nu-1}(x)\,J_\nu(x) \\[0.3em]
	J_\nu(x)^2-J_{\nu-1}(x)^2
\end{pmatrix}\,,
\]
\[
\Psi_\kappa(x)
=
\begin{pmatrix}
	\Psi_{\kappa,1}(x)\\[0.3em]
	\Psi_{\kappa,2}(x)
\end{pmatrix}
=
\frac{\pi x}{2}
\begin{pmatrix}
	J_{\nu-1}(x)Y_\nu(x)+J_\nu(x)Y_{\nu-1}(x) \\[0.3em]
	J_{\nu-1}(x)Y_{\nu-1}(x)-J_\nu(x)Y_\nu(x)
\end{pmatrix}\,,
\]
with $\nu=\kappa+\tfrac12\,$. With this notation, the previous computations can be summarized in the following proposition.

\begin{prop}\label{prop:prelim-identities}
	Let $(v_1,v_2)$ satisfy the linearized spectral conditions
	\eqref{eq:v1-system}--\eqref{eq:v2-system} together with the constraint
	\eqref{eq:zero-eigenvalue-constraintbis}. Then the following statements hold.
	
	\begin{enumerate}
		
		\item The function $v_1$ satisfies
		\begin{equation}\label{eq:prop-v1-constraint}
			\int_0^1 (1-2x^{2\nu})\,v_1(x)\,dx =0\, .
		\end{equation}
		Moreover, for all $z\in\C^*\,$,
		\begin{equation}\label{eq:prop-v1}
			\int_0^1
			\bigl[J_\nu(z)\,\Psi_{\kappa,1}(zx)
			+
			Y_\nu(z)\,\Phi_{\kappa,1}(zx)\bigr]
			\,v_1(x)\,dx
			=0 \,.
		\end{equation}
		
		\item For all $z\in\C^*$, the function $v_2$ satisfies
		\begin{equation}\label{eq:prop-v2}
			\int_0^1
			\bigl[J_\nu(z)\,\Psi_{\kappa,2}(zx)
			+
			Y_\nu(z)\,\Phi_{\kappa,2}(zx)\bigr]
			\,v_2(x)\,dx
			=0 \, .
		\end{equation}
		
	\end{enumerate}
\end{prop}

\subsection{A first injectivity result}

We have seen in the previous subsection that the linearization condition implies 
the integral constraint
\begin{equation}\label{eq:v1-constraintconc}
	\int_0^1 \bigl(1-2x^{2\nu}\bigr)\,v_1(x)\,dx = 0 \,,
	\qquad \nu = \kappa + \frac12 ,
\end{equation}
while the presence of the zero eigenvalue imposes (see \eqref{eq:single-ell-kernel})
\begin{equation}\label{eq:v2-constraintconc}
	\int_0^1 x^{2\nu -1}\,v_2(x)\,dx = 0 \,.
\end{equation}

\vspace{0.2cm}\noindent
Our first injectivity result for the Fr\'echet differential in the AKNS setting
is based on the classical M\"untz--Sz\'asz theorem
\cite{Muntz12,Muntz14,Szasz16}.

\begin{thm}\label{thm:muntz-AKNS}
Let $(v_1,v_2)\in L^2(0,1)^2$ be a real-valued vector function satisfying
the AKNS linearized constraints \eqref{eq:v1-constraintconc}--\eqref{eq:v2-constraintconc}
for an infinite increasing sequence $\{\kappa_k\}_{k\ge1}\subset\mathbb{N}^*$\,, with 
$\nu_k := \kappa_k + \tfrac12\, .$
Assume moreover that
\[
\sum_{k=1}^\infty \frac{1}{\kappa_k} = +\infty\,.
\]
Then $(v_1,v_2)=(0,0)$ almost everywhere in $(0,1)\,$.
\end{thm}

\begin{proof}
From the identity
\begin{equation}\label{eq:moment-nu-k}
\int_0^1 \bigl(1-2x^{2\nu_k}\bigr)\,v_1(x)\,dx = 0\,,
\end{equation}
we let $k\to\infty$ and we get  
\[
\int_0^1 v_1(x)\,dx = 0\,.
\]
Substituting this back into \eqref{eq:moment-nu-k}, we obtain the moment identities
\[
\int_0^1 x^{2\nu_k}\, v_1(x)\, dx = 0\,,
\qquad \text{for all } k\ge1\,.
\]
Since 
\[
\sum_{k=1}^\infty \frac{1}{\kappa_k}=+\infty\,,
\]
the classical M\"untz-Sz\'asz theorem applies to the family
$\{x^{2\nu_k}\}_{k\ge1}$ on $(0,1)$ and implies that  
\[
v_1 = 0 \quad \text{almost everywhere on }(0,1)\,.
\]
The argument for $v_2$ is identical, using the constraint 
\eqref{eq:v2-constraintconc}.
Hence $(v_1,v_2)=(0,0)$ a.e., which completes the proof.
\end{proof}

\subsection{Transformation operators and Green's identity}

We recall the definition of the transformation operators introduced in the work of Serier.
Such operators first appeared in the seminal paper of Guillot and Ralston
\cite{GR} in connection with the inverse spectral problem for the radial 
Schr\"odinger operator (the case $\kappa=1$). They were later extended to general
integer $\kappa$ by Rundell and Sacks \cite{RuSa01}, and subsequently refined in 
\cite{Ser07}. 

\medskip
\noindent
In the AKNS setting, Serier constructed similar operators adapted to the
first-order matrix structure. A key difference with the Schr\"odinger case is
that the inverse operators have a more favorable structure.

\medskip
\noindent
Throughout this subsection we use the vector-valued functions
$\Phi_\kappa$ and $\Psi_\kappa$ introduced in the previous section, and
we keep the notation
\[
\Phi_\kappa(x)=
\begin{pmatrix}
	\Phi_{\kappa,1}(x)\\[0.2em]
	\Phi_{\kappa,2}(x)
\end{pmatrix}\,,
\qquad
\Psi_\kappa(x)=
\begin{pmatrix}
	\Psi_{\kappa,1}(x)\\[0.2em]
	\Psi_{\kappa,2}(x)
\end{pmatrix}\,.
\]

\medskip\noindent
The next  lemma is taken from \cite{Serier2006} and will be essential
for analyzing the inverse problem in the AKNS setting. First, let us give some notation\footnote{We adopt the same notation as that introduced by Serier~\cite{Serier2006}.}.
\begin{nota} For all $n\in\N\,$, let $U_n$ and
$V_n$ be defined by
$$U_n(x)=\Bigg[\begin{array}{c}
  0 \\
  x^n \\
\end{array}\Bigg]\quad\mathrm{ and }\quad
V_n(x)=\Bigg[\begin{array}{c}
  x^n \\
  0 \\
\end{array}\Bigg]\quad x\in[0,1]\,.$$
\end{nota}

\begin{lemma}\label{akns_lemme_reduction_indice}
For each $\kappa\in\mathbb{N}\,$, define the operator
\[
S_{\kappa+1}:\;L^2(0,1)^2\longrightarrow L^2(0,1)^2\,,
\qquad
S_{\kappa+1}[p,q]:=\bigl(S_{\kappa,1}[p],\,S_{\kappa,2}[q]\bigr)\,,
\]
where
\[
S_{\kappa,1}[p](x)=p(x)-2(2\kappa+1)x^{2\kappa}\!\int_x^1\frac{p(t)}{t^{2\kappa+1}}dt\,,
\qquad
S_{\kappa,2}[q](x)=q(x)-2(2\kappa+1)x^{2\kappa+1}\!\int_x^1\frac{q(t)}{t^{2\kappa+2}}dt\,.
\]
We also set $S_0:=\mathrm{Id}$.
The operators $\{S_\kappa\}_{\kappa\ge0}$ satisfy:
\begin{enumerate}[(i)]
\item The adjoint is given by
\[
S_{\kappa+1}^\ast[f,g]=\bigl(S_{\kappa,1}^\ast[f],\,S_{\kappa,2}^\ast[g]\bigr)\,,
\]
with
\[
S_{\kappa,1}^\ast[f](x)
= f(x)-\frac{2(2\kappa+1)}{x^{2\kappa+1}}\int_0^x t^{2\kappa}f(t)\,dt\,,
\qquad
S_{\kappa,2}^\ast[g](x)
= g(x)-\frac{2(2\kappa+1)}{x^{2\kappa+2}}\int_0^x t^{2\kappa+1}g(t)\,dt\,.
\]

\item The family $\{S_\kappa\}$ is commuting:
\[
S_{\kappa}S_{m}=S_{m}S_{\kappa}\qquad\forall\,\kappa,m\in\mathbb{N}\,.
\]

\item Each $S_\kappa$ is bounded on $L^2(0,1)^2$\,.

\item With $N_{\kappa+1}:=\ker S_{\kappa+1}^\ast$, one has
\[
N_{\kappa+1}=\mathrm{Vect}(U_{2\kappa},\,V_{2\kappa+1})\,.
\]

\item $S_{\kappa+1}$ is an isomorphism from $L^2(0,1)^2$ onto 
$N_{\kappa+1}^{\perp}$, with inverse
\[
A_{\kappa+1}[f,g]:=\bigl(S_{\kappa,2}^\ast[f],\,S_{\kappa,1}^\ast[g]\bigr)\,.
\]

\item The functions $\Phi_\kappa$ and $\Psi_\kappa$ satisfy the reduction relations
\[
\Phi_{\kappa+1}=-S_{\kappa+1}^\ast[\Phi_\kappa]\,,
\qquad
\Psi_{\kappa+1}=-S_{\kappa+1}^\ast[\Psi_\kappa]\,.
\]
\end{enumerate}
\end{lemma}

\vspace{0.3cm}\noindent
We will also need the following complementary result which is analogous
to Lemma 3.4 in \cite{RuSa01}. 

\vspace{0.2cm}

\begin{lemma}\label{ODESl}
Let $\kappa\ge 0$ and let $f,g\in L^2(0,1)$\,. Then:

\begin{enumerate}
\item If $g = S_{\kappa,1}[f]\,$, then in the sense of distributions on $(0,1)\,$,
\begin{equation}\label{eq:ODEsl1}
g^{(2\kappa+1)}(x)
= \frac{4\kappa+2}{x}\,f^{(2\kappa)}(x) + f^{(2\kappa+1)}(x)\,.
\end{equation}

\item If $g = S_{\kappa,2}[f]\,$, then in the sense of distributions on $(0,1)$,
\begin{equation}\label{eq:ODEsl2}
g^{(2\kappa+2)}(x)
= \frac{4\kappa+2}{x}\,f^{(2\kappa+1)}(x) + f^{(2\kappa+2)}(x)\,.
\end{equation}
\end{enumerate}
\end{lemma}

\begin{proof}
We adapt the argument of \cite{RuSa01} for the first identity \eqref{eq:ODEsl1}.  
Starting from
\begin{equation}\label{eq:g}
g(x)
= f(x)
- 2(2\kappa+1)\, x^{2\kappa} \int_x^{1} s^{-2\kappa-1} f(s)\,ds \,,
\end{equation}
a single differentiation yields
\begin{equation}\label{eq:gprime}
g'(x)
= f'(x)
- 4\kappa(2\kappa+1)\, x^{2\kappa-1}
    \int_x^{1} s^{-2\kappa-1} f(s)\, ds
+ \frac{2(2\kappa+1)}{x} f(x)\,.
\end{equation}
To eliminate the integral term, consider
\[
2\kappa\,\eqref{eq:g} \;-\; x \eqref{eq:gprime}\,,
\]
which gives
\[
2\kappa g(x) - x g'(x)
= - (2\kappa+2) f(x) - x f'(x)\,.
\]
Differentiating once more,
\[
(2\kappa-1) g'(x) - x g''(x)
= - (2\kappa+3) f'(x) - x f''(x)\,,
\]
and by iterating this procedure $k$ times one obtains
\[
(2\kappa - k)\, g^{(k)}(x) - x g^{(k+1)}(x)
= - (2\kappa+k+2)\, f^{(k)}(x) - x f^{(k+1)}(x)\,.
\]
Setting $k = 2\kappa$ and dividing by $x$ yields exactly \eqref{eq:ODEsl1}.  
The proof of \eqref{eq:ODEsl2} is entirely analogous.
\end{proof}

\vspace{0.2cm}\noindent
We now consider the composite operator \(T_\kappa\), obtained by composing the index-reduction operators \(S_1,\dots,S_\kappa\,\), which carries Bessel kernels to trigonometric ones.

\medskip

\begin{lemma}\label{akns-akns_lemme_bessel_sinus}
For every $\kappa\in\mathbb{N}\,$, define
\[
T_\kappa=(-1)^{\kappa+1}S_\kappa S_{\kappa-1}\cdots S_1\,,
\qquad
T_0:=-S_0\,.
\]
Write $T_\kappa[p,q]=(T_\kappa^1[p],\,T_\kappa^2[q])\,$. Then:
\begin{enumerate}[(i)]
\item $T_\kappa$ is bounded and injective, and for all $p,q$ and all $\lambda\in\mathbb{C}\,$,
\[
\int_0^1\Phi_\kappa(\lambda t)\cdot\binom{p(t)}{q(t)}\,dt
=
\int_0^1
\binom{\sin(2\lambda t)}{\cos(2\lambda t)}\cdot T_\kappa[p,q](t)\,dt\,,
\]
\[
\int_0^1\Psi_\kappa(\lambda t)\cdot\binom{p(t)}{q(t)}\,dt
=
\int_0^1
\binom{\cos(2\lambda t)}{-\sin(2\lambda t)}\cdot T_\kappa[p,q](t)\,dt\,.
\]

\item The adjoint $T_\kappa^\ast$ satisfies 
\[
\Phi_\kappa(\lambda x)
= T_\kappa^\ast
\binom{\sin(2\lambda \cdot)}{\cos(2\lambda \cdot)}(x),
\qquad
\Psi_\kappa(\lambda x)
= T_\kappa^\ast
\binom{\cos(2\lambda \cdot)}{-\sin(2\lambda \cdot)}(x).
\]

and
\[
\ker T_\kappa^\ast=\bigoplus_{k=1}^\kappa N_k\,.
\]

\item $T_\kappa$ defines an isomorphism from $L^2(0,1)^2$ onto 
$\bigl(\bigoplus_{k=1}^\kappa N_k\bigr)^\perp$, with inverse
\[
B_\kappa[f,g]=\bigl((T_\kappa^2)^\ast[f],\,(T_\kappa^1)^\ast[g]\bigr)\,.
\]
\end{enumerate}
\end{lemma}

\begin{rem}\label{separation}
Taking, for instance, \((p,q)=(p,0)\) in Lemma~\ref{akns-akns_lemme_bessel_sinus}(i),
we obtain that, for every \(p\) and every \(\lambda\in\mathbb{C}\),
\[
\int_0^1 \Phi_{\kappa,1}(\lambda t)\, p(t)\,dt
=
\int_0^1 \sin(2\lambda t)\, T_{\kappa}^1(p)(t)\,dt\,.
\]
\end{rem}


\vspace{0.4cm}\noindent
We now apply Lemma~\ref{akns-akns_lemme_bessel_sinus}(i).  
Using the classical identity
\[
Y_\nu(x)=(-1)^{\,\kappa+1}\,J_{-\nu}(x)\,,
\]
Proposition~\ref{prop:prelim-identities} can be rewritten in the following equivalent form: for all $z\in\C^*\,$,
\begin{equation}\label{eq:system-Tell}
\left\{
\begin{aligned}
&\int_0^1\Bigl[
J_\nu(z)\cos(2zx)\,T_\kappa^1[v_1](x)
+(-1)^{\kappa+1}J_{-\nu}(z)\sin(2zx)\,T_\kappa^1[v_1](x)
\Bigr]\,dx=0\,,\\[0.6em]
&\int_0^1\Bigl[
-\,J_\nu(z)\sin(2zx)\,T_\kappa^2[v_2](x)
+(-1)^{\kappa+1}J_{-\nu}(z)\cos(2zx)\,T_\kappa^2[v_2](x)
\Bigr]\,dx=0\,.
\end{aligned}
\right.
\end{equation}


\vspace{0.2cm}\noindent
For later use, we recall the explicit formulas for Bessel functions of
half-integer order together with the associated polynomials introduced
in \cite[10.1.19--20]{AS64}.  
When $\kappa=0,1,2,\dots$ and $z\in\C$, one has the classical representations
\begin{align}
J_{\kappa+\tfrac12}(z)
&=
\sqrt{\frac{2}{\pi z}}\,
\Bigl(
   P_\kappa\!\left(\tfrac1z\right)\,\sin z
   \;-\;
   Q_{\kappa-1}\!\left(\tfrac1z\right)\,\cos z
\Bigr)\,,
\label{besseldemientier1}
\\[0.3em]
J_{-\kappa-\tfrac12}(z)
&=
(-1)^{\kappa}\sqrt{\frac{2}{\pi z}}\,
\Bigl(
   P_\kappa\!\left(\tfrac1z\right)\,\cos z
   \;+\;
   Q_{\kappa-1}\!\left(\tfrac1z\right)\,\sin z
\Bigr)\,.
\label{besseldemientier2}
\end{align}

\vspace{0.15cm}\noindent
The polynomials \(P_\kappa\) and \(Q_\kappa\), each of degree~\(\kappa\), are generated by the three-term recurrences 
\begin{align}
P_{\kappa+1}(t)
&=(2\kappa+1)\,t\,P_\kappa(t)-P_{\kappa-1}(t)\,,
\qquad \kappa\ge1\,,
\\
Q_{\kappa+1}(t)
&=(2\kappa+3)\,t\,Q_\kappa(t)-Q_{\kappa-1}(t)\,,
\qquad \kappa\ge0\,,
\end{align}
with initial values
\[
P_0(t)=1\,,\qquad P_1(t)=t\,,
\qquad
Q_{-1}(t)=0\,,\qquad Q_0(t)=1\,.
\]

\medskip\noindent
Observe that \(P_\kappa\) and \(Q_\kappa\) inherit the parity of~\(\kappa\):
they are even functions when \(\kappa\) is even and odd functions when \(\kappa\) is odd.

\vspace{0.15cm}\noindent
For illustration, the lowest half-integer orders give
\begin{equation}\label{bessel12}
J_{\tfrac12}(z)=\sqrt{\frac{2}{\pi z}}\,\sin z\,,
\qquad
J_{-\tfrac12}(z)=\sqrt{\frac{2}{\pi z}}\,\cos z\,.
\end{equation}
The next pair is
\begin{equation}\label{bessel32}
J_{\tfrac32}(z)
=\sqrt{\frac{2}{\pi z}}\Bigl(\frac{\sin z}{z}-\cos z\Bigr)\,,
\qquad
J_{-\tfrac32}(z)
=\sqrt{\frac{2}{\pi z}}\Bigl(-\frac{\cos z}{z}-\sin z\Bigr)\,.
\end{equation}

\vspace{0.15cm}\noindent
Using the recurrence relation, the first few polynomials are
\begin{equation}\label{polynome}
\begin{aligned}
P_0(t)&=1\,,         &\qquad Q_{-1}(t)&=0\,,\\
P_1(t)&=t\,,         &\qquad Q_0(t)&=1\,,\\
P_2(t)&=3t^2-1\,,    &\qquad Q_1(t)&=3t\,,\\
P_3(t)&=15t^3-6t\,,  &\qquad Q_2(t)&=15t^2-1\,.
\end{aligned}
\end{equation}

\vspace{0.2cm}\noindent
Gathering the previous identities, we arrive at the following statement.

\begin{prop}\label{lelemmedebase}
Assume that for $\kappa \in \N$,
\[
D_{(p,q)}\lambda_{\kappa,n} (0,0)\cdot (v_1, v_2)=0\,,
\qquad\text{for all } n\in\mathbb{Z}\,.
\]
Then, for every \(z\in\C\) and every integer \(\kappa\ge0\)\,, one obtains  the following  identity: for all $z\in\C^*$,
\begin{equation}\label{eq:system-integraledebase}
\left\{
\begin{aligned}
&\int_0^1
\Bigl[
\Bigl(
P_\kappa\!\bigl(\tfrac1z\bigr)\,\sin\!\bigl(z(2x-1)\bigr)
+
Q_{\kappa-1}\!\bigl(\tfrac1z\bigr)\,\cos\!\bigl(z(2x-1)\bigr)
\Bigr)
\,T_\kappa^1[v_1](x)
\Bigr]\,dx
=0,\\[0.6em]
&\int_0^1
\Bigl[
\Bigl(
P_\kappa\!\bigl(\tfrac1z\bigr)\,\cos\!\bigl(z(2x-1)\bigr)
-
Q_{\kappa-1}\!\bigl(\tfrac1z\bigr)\,\sin\!\bigl(z(2x-1)\bigr)
\Bigr)
\,T_\kappa^2[v_2](x)
\Bigr]\,dx
=0.
\end{aligned}
\right.
\end{equation}

\end{prop}

\begin{proof}
The identity follows directly from \eqref{eq:system-Tell} together with
the half-integer representations \eqref{besseldemientier1}-\eqref{besseldemientier2}, 
after rewriting the products of Bessel functions using 
elementary trigonometric relations. 
\end{proof}

\medskip\noindent
We now introduce the sequence of polynomials
\( \{A_\kappa(t)\}_{\kappa \in \mathbb{N}} \)\,,
defined recursively by
\[
A_0(t) = 1,
\qquad
A_1(t) = 1 - \frac{t}{2}\,,
\]
and, for all \( \kappa \geq 1 \)\,,
\[
A_{\kappa+1}(t)
= (2\kappa + 1)\,A_\kappa(t)
+ \frac{t^2}{4}\,A_{\kappa-1}(t)\,.
\]

\begin{rem}
The first polynomials of the sequence beyond \(A_1\) are explicitly given by
\[
A_2(t) = \tfrac{1}{4}t^2 - \tfrac{3}{2}t + 3\,,
\qquad
A_3(t) = -\tfrac{1}{8}t^3 + \tfrac{3}{2}t^2
        - \tfrac{15}{2}t + 15\,.
\]
\end{rem}

\medskip\noindent
The second equation of the system~\eqref{eq:system-integraledebase}
coincides with the equation already studied in
\cite[Proposition~5.1]{GGHN2025}\footnote{In the case $q=-m$, where $m$ is a constant interpreted as a mass, the AKNS system is closely related to a scalar Schr\"odinger equation (see \cite[Eq.~(1.4)]{AHM} and Appendix~A (\emph{Open problems}) of the present paper). Consequently, the analysis reduces to a second-order Schr\"odinger-type problem already studied in \cite{GGHN2025}.}.
We may therefore directly invoke
\cite[Theorem~6.6]{GGHN2025}.
The first equation of the system~\eqref{eq:system-integraledebase}
can be handled in the same way, by closely following the proof of
\cite[Theorem~6.6]{GGHN2025}. We therefore
obtain the following result, where \( D = \frac{d}{dx} \)\,.

\begin{thm}\label{mainthm}
	Let $(v_1, v_2) \in L^2(0,1)^2 \,$. Assume that, for some \( \kappa \in \N \),
	\[
	D_{(p,q)} \lambda_{\kappa,n}(0,0)\cdot (v_1, v_2) = 0,
	\qquad \text{for all } n \in \mathbb{Z}.
	\]
	Then, in the sense of distributions, the functions
	\[
	A_\kappa\bigl(D\bigr)\bigl[T_\kappa^{j}[v_j]\bigr],
	\qquad j \in \{1,2\},
	\]
	are even for \( j = 1 \) and odd for \( j = 2 \)
	with respect to the midpoint \( x = \tfrac{1}{2} \).
\end{thm}


\section{Kernel  of the Fr\'echet differential}

\subsection{Injectivity of the differential for the pair $(\kappa_1, \kappa_2)=(0,1)$ }

In this subsection, we assume that the perturbation
\( (v_1, v_2) \) satisfies the linearized spectral condition
\[
D_{(p,q)} \lambda_{\kappa,n}(0,0) \cdot (v_1, v_2) = 0\,,
\qquad n \in \mathbb{Z},
\]
for both effective angular momenta \( \kappa = 0 \) and \( \kappa = 1 \)\,.

\medskip\noindent
For \( \kappa = 0 \), we already know that \(v_1\) is even and \(v_2\) is odd about \(x=\tfrac12\).
We now apply Theorem~\ref{mainthm} with \( \kappa = 1 \), which yields that
\(A_1(D)\bigl[T_1^{j}[v_j]\bigr]\) is even for \(j=1\) and odd for \(j=2\)\,,
with respect to the same midpoint.

\medskip\noindent
We begin with the simpler case \( j = 2 \)\,.
A straightforward computation yields
\begin{equation}\label{eqTv2}
2 A_1(D)\bigl[T_1^{2}[v_2]\bigr](x) = 2 A_1(D)\bigl[S_{0,2}[v_2]\bigr](x)
= -v_2'(x)
+ \Bigl(2 - \frac{2}{x}\Bigr)\, v_2(x)
- (4x - 2)\int_x^1 \frac{v_2(t)}{t^2}\,dt\, .
\end{equation}
Setting \( y(x) := v_2'(x) \) and evaluating \eqref{eqTv2} at \( x = \tfrac{1}{2} \)\,,
we obtain \( y(\tfrac{1}{2}) = 0 \)\,, since \( v_2 \) is odd. We further compute
\begin{equation}\label{eq:G-definition}
\begin{aligned}
G(x)
&:= D^2 A_1(D)\bigl[T_1^{2}[v_2]\bigr](x) \\
&= A_1(D)\bigl[D^2(S_{0,2}[v_2])\bigr](x) \\
&= A_1(D)\bigl[\tfrac{2}{x}y + y'\bigr](x) \\
&= -\tfrac12 y''(x)
+ \Bigl(1 - \frac{1}{x}\Bigr)y'(x)
+ \Bigl(\frac{2}{x} + \frac{1}{x^2}\Bigr)y(x)\,.
\end{aligned}
\end{equation}
where we have used  Lemma~\ref{ODESl} (2) in the third line. Recalling that \(G\) is odd, the identity
\[
G(x) + G(1-x) = 0\,,
\]
holds for all \(x \in (0,1)\)\,.
Since \(y\) is even, this identity implies that \(y\) satisfies a linear
second-order differential equation on \( (0,1) \), together with the conditions

\[
y\!\left(\tfrac{1}{2}\right) = 0\,,
\qquad
y'\!\left(\tfrac{1}{2}\right) = 0\,.
\]
By the Cauchy--Lipschitz theorem, we conclude that \( y \equiv 0 \)\,.
Therefore \( v_2' = y = 0 \), so \( v_2 \) is constant.
Since \( v_2 \) is odd with respect to \( x = \tfrac{1}{2} \)\,,
this constant must vanish, and thus $v_2 \equiv 0$\,.

\vspace{0.3cm}\noindent
We now examine the case \(j=1\).
Using Lemma~\ref{ODESl} (1), a straightforward computation yields
\begin{equation}\label{eq:Gj1}
G(x)
:= D A_1(D)\bigl[T_1^1[v_1]\bigr](x)
= -\tfrac12 v_1''(x)
+ \Bigl(1-\frac{1}{x}\Bigr)v_1'(x)
+ \Bigl(\frac{2}{x}+\frac{1}{x^{2}}\Bigr)v_1(x)\,.
\end{equation}
The function \(G\) is odd. Writing \(G(x)+G(1-x)=0\) and using the fact that
\(v_1\) is even, we infer that \(v_1\) satisfies the following second--order
linear ordinary differential equation on \((0,1)\)\,:
\begin{equation}\label{eq:v1-ODE}
v_1''(x)
+ \Bigl(\frac{1}{x} - \frac{1}{1-x}\Bigr) v_1'(x)
- \Bigl(\frac{2}{x} + \frac{2}{1-x}
       + \frac{1}{x^{2}} + \frac{1}{(1-x)^{2}}\Bigr) v_1(x)
= 0\,,
\qquad x\in(0,1)\,.
\end{equation}
We now assume that there exists a solution \(v_1\) of~\eqref{eq:v1-ODE} which is even with respect to the midpoint \(x=\tfrac12\), and we impose the normalization conditions
\[
v_1\!\left(\tfrac12\right)=1\,,
\qquad
v_1'\!\left(\tfrac12\right)=0\,.
\]
Using \textsc{Mathematica}, we obtain the explicit closed-form expression
\begin{equation}\label{eq:v1-explicita}
v_1(x)
=
\frac{2x^{2}-2x+1}{2x(1-x)}
=
\frac12\left(\frac{1}{x}+\frac{1}{1-x}\right)-1\,,
\qquad x\in(0,1)\,.
\end{equation}
In particular, \(v_1\) blows up like \(\frac{1}{2x}\) as \(x\to0^+\) and therefore does not belong to \(L^2(0,1)\).
It follows that one must have \(v_1\!\left(\tfrac12\right)=0\), and the Cauchy--Lipschitz theorem then implies that
\(v_1 \equiv 0\) on \((0,1)\).

\medskip\noindent
Combining the conclusions of the two cases $j=1$ and $j=2$, we obtain that
\[
(v_1,v_2)=(0,0).
\]
Hence the kernel of the Fr\'echet differential of the spectral map at the zero
potential is trivial for the pair of effective angular momenta $(\kappa_1,\kappa_2)=(0,1)$.
We have therefore proved the following result.

\begin{thm}[Injectivity for the pair $(0,1)$]
	\label{thm:injective-01a}
	For $(\kappa_1,\kappa_2)=(0,1)$, the Fr\'echet differential of the spectral map
	\[
	D_{(p,q)}\mathcal S_{0,1}(0,0)
	:
	L^2(0,1)\times L^2(0,1)
	\longrightarrow
	\ell^2_{\mathbb R}(\mathbb Z)\times \ell^2_{\mathbb R}(\mathbb Z)
	\]
	is one to one.
\end{thm}

\subsection{Injectivity of the differential for the pair $(\kappa_1, \kappa_2)=(0,2)$}

\medskip\noindent
Throughout this subsection, we assume that the perturbation \((v_1,v_2)\) fulfills the
linearized spectral constraint
\[
D_{(p,q)}\lambda_{\kappa,n} (0,0)\cdot (v_1,v_2)=0\,,
\qquad n\in\mathbb{Z}\,,
\]
simultaneously for the  effective angular momenta \(\kappa=0\) and \(\kappa=2\)\,.

\medskip\noindent
For \(\kappa=0\), as before, \(v_1\) is even and \(v_2\) is odd with respect to the midpoint \(x=\tfrac12\)\,.

\medskip\noindent
Let us now examine the case \(\kappa = 2\)\,. According to Theorem~\ref{mainthm}, and setting \(w_j := T_2^{\,j}[ v_j ]\),
\(j=1,2\)\,, we have
\begin{equation}\label{eq:schrodinger_condition}
A_2(D)\, \bigl[w_j \bigr]
\quad \text{is even if \(j=1\)\,, and odd if \(j=2\)\,.}
\end{equation}
Here \(A_2(t)=\tfrac14 t^{2}-\tfrac32 t+3\) and \(D=\tfrac{d}{dx}\)\,.

\vspace{0.3cm}\noindent
We begin by studying the case \(j=2\). We introduce the following notation.  
Set $f = S_{0,2}\, [v_2 ]$ \,,  so that $w_2 = -\, S_{1,2} [f]\,$.  
Differentiating \eqref{eq:schrodinger_condition} four times  and applying Lemma~\ref{ODESl}\,(2) with $\kappa=1$\,, we obtain
\begin{equation}\label{eq:A2-after-derivatives}
A_2(D)\!\left( f^{(4)}(x) + \frac{6}{x} f^{(3)}(x) \right)
\quad \text{is odd.}
\end{equation}

\medskip\noindent
On the other hand, since $f = S_{0,2}\, [v_2]$, a second application of Lemma~\ref{ODESl}\,(2), now with $\kappa = 0$, yields
\begin{equation}\label{eq:def-f}
f''(x) = \frac{2}{x}\, v_2'(x) + v_2''(x)\,.
\end{equation}
\noindent
Setting $y(x):=v_2'(x)$ (which is even), we obtain after simplification
\begin{equation}
\begin{aligned}
G(x)
&:= 4A_2(D)\!\left(f^{(4)}(x)+\frac{6}{x}f^{(3)}(x)\right) \\
&= y^{(5)}(x)
+ \left(\frac{8}{x}-6\right) y^{(4)}(x)
+ \left(12 - \frac{48}{x} - \frac{8}{x^{2}}\right) y^{(3)}(x) \\
&\quad
+ \left(\frac{96}{x} - \frac{24}{x^{3}}\right) y''(x)
+ \left(\frac{96}{x^{2}} + \frac{144}{x^{3}} + \frac{96}{x^{4}}\right) y'(x) \\
&\quad
+ \left(-\frac{96}{x^{3}} - \frac{144}{x^{4}} - \frac{96}{x^{5}}\right) y(x)\,.
\end{aligned}
\end{equation}

\medskip\noindent
Because \eqref{eq:schrodinger_condition} asserts that $A_2(D) [w_2 ]$ is odd, 
and differentiation four times does not alter odd parity, 
we conclude that $G(x)$ is itself odd. Writing $G(x)+G(1-x)=0$ and using the fact that $y$ is even, we see that $y$
satisfies a linear differential equation of order~$4$. We denote by 

\begin{equation}
v_2(x) := \int_{1/2}^x y(t)\,dt
\end{equation}
the unique odd primitive of $y(x)$. An immediate computation gives the following expression:
\begin{equation}
w_2 (x)= T_2^2 [v_2](x)
= -\,S_{1,2}\bigl[ S_{0,2}[v_2]\bigr](x)
= -\,v_2(x)
- 4x \int_x^1 \frac{v_2(t)}{t^{2}}\,dt
+ 12x^{3} \int_x^1 \frac{v_2(t)}{t^{4}}\,dt \,.
\end{equation}
Applying the differential operator $4A_2(D)=D^{2}-6D+12$ to $w_2(x)$, we obtain
\begin{equation}\label{eq:A2w2-explicit}
\begin{aligned}
4A_2(D)\bigl[w_2\bigr](x)
=\;&
-\,v_2''(x)
+\Bigl(6-\frac{8}{x}\Bigr)v_2'(x)
+\Bigl(-12+\frac{48}{x}-\frac{24}{x^{2}}\Bigr)v_2(x)
\\[0.4em]
&\quad
+\bigl(24-48x\bigr)
\int_x^1 \frac{v_2(t)}{t^{2}}\,dt
\\[0.4em]
&\quad
+\bigl(72x-216x^{2}+144x^{3}\bigr)
\int_x^1 \frac{v_2(t)}{t^{4}}\,dt \,.
\end{aligned}
\end{equation}
Evaluating this expression at $x=\tfrac12$ and using that $v_2$ is odd, we obtain
\begin{equation}\label{eq:A2w2-half}
4A_2(D)\bigl[w_2\bigr]\!\left(\tfrac12\right)
=
-\,v_2''\!\left(\tfrac12\right)
-10\,v_2'\!\left(\tfrac12\right)
-12\,v_2\!\left(\tfrac12\right)
= -10\,y\!\left(\tfrac12\right)\,.
\end{equation}
Since $4A_2(D)\bigl[w_2\bigr](x)$ is odd, we therefore conclude that $y\!\left(\tfrac12\right)=0$\,.

\vspace{0.3cm}\noindent
Proceeding in the same way, we compute
\begin{equation}\label{eq:D2A2w2-explicit}
\begin{aligned}
4\,D^{2}A_{2}(D)\bigl[w_{2}\bigr](x)
=\;&
-\,v_2^{(4)}(x)
+\Bigl(6-\frac{8}{x}\Bigr)v_2^{(3)}(x)
+\Bigl(-12+\frac{48}{x}-\frac{8}{x^{2}}\Bigr)v_2''(x)
\\[0.4em]
&\quad
+\Bigl(-\frac{96}{x}+\frac{96}{x^{2}}+\frac{8}{x^{3}}\Bigr)v_2'(x)
+\Bigl(-\frac{288}{x^{2}}+\frac{144}{x^{3}}\Bigr)v_2(x)
\\[0.4em]
&\quad
+\bigl(-432+864x\bigr)
\int_x^1 \frac{v_2(t)}{t^{4}}\,dt\, .
\end{aligned}
\end{equation}
Recalling that $y(x)=v_2'(x)$ is even and $y\!\left(\tfrac12\right)=0$, we evaluate at $x=\tfrac12$ and obtain 
\[
4\,D^{2}A_{2}(D)\bigl[w_{2}\bigr]\!\left(\tfrac12\right)
=
-\,y^{(3)}\!\left(\tfrac12\right)
-10\,y''\!\left(\tfrac12\right)
+52\,y'\!\left(\tfrac12\right)
+256\,y\!\left(\tfrac12\right)
= -10\,y''\!\left(\tfrac12\right)\,.
\]
Since $4\,D^{2}A_{2}(D)\bigl[w_{2}\bigr](x)$ is also an odd function, we conclude, as above, that $ y''\!\left(\tfrac12\right)=0$\,.

\medskip\noindent
In conclusion, $y$ satisfies a fourth--order linear differential equation with
the initial conditions
\begin{equation}\label{eq:CLzzz}
y\!\left(\tfrac12\right)=0\,,\qquad
y'\!\left(\tfrac12\right)=0\,,\qquad
y''\!\left(\tfrac12\right)=0\,,\qquad
y^{(3)}\!\left(\tfrac12\right)=0\,.
\end{equation}
The Cauchy--Lipschitz theorem then implies that $y\equiv 0\,$.
Since $y=v_2'$ and $v_2$ is odd, this in turn forces $v_2\equiv 0$\,.

\vspace{0.5cm}\noindent
We now turn to the analysis in the case $j=1$. In this case, $A_2(D)\, [w_1]$ is an even function.

\vspace{0.2cm}\noindent
We introduce
\[
w_1(x):=-S_{1,1} \bigl[S_{0,1} [v_1 ]\bigr](x)
=
-\,v_1(x)
-4\int_x^1 \frac{v_1(t)}{t}\,dt
+12x^{2}\int_x^1 \frac{v_1(t)}{t^{3}}\,dt\,.
\]
A direct computation yields
\begin{equation}\label{eq:DA2w1}
\begin{aligned}
D A_2(D) [w_1](x)
&=
-\frac14\,v_1^{(3)}(x)
+\Bigl(\frac32-\frac{2}{x}\Bigr)v_1''(x)
+\Bigl(\frac{12}{x}-\frac{2}{x^2}-3\Bigr)v_1'(x)
\\
&\quad
+\Bigl(\frac{2}{x^3}+\frac{24}{x^2}-\frac{24}{x}\Bigr)v_1(x)
+
36(2x-1)\int_x^1\frac{v_1(t)}{t^3}\,dt\,.
\end{aligned}
\end{equation}
The function $D A_2(D) [w_1]$ is odd with respect to $x=\tfrac12\,$. Moreover, in the case
$\kappa=0$, we recall that $v_1$ is even. Evaluating \eqref{eq:DA2w1} at $x=\tfrac12\,$,
we have $v_1'(\tfrac12)=v_1^{(3)}(\tfrac12)=0\,$, and the integral term vanishes since
$2x-1=0$. Therefore,
\begin{equation}\label{eq:DA2w1-midpoint}
64\,v_1\!\left(\tfrac12\right)
-\frac52\,v_1''\!\left(\tfrac12\right) =0\,.
\end{equation}

\medskip\noindent
Now, following the usual convention, we introduce
\[
f = S_{0,1} [v_1]\,,
\qquad\text{so that}\qquad
w_1 = -\, S_{1,1} [f]\,.
\]
After differentiating  three times $A_2(D)\, [w_1]$ and invoking Lemma~\ref{ODESl} (1) with $\kappa=1$, we arrive at
\begin{equation}\label{eq:A2-after-derivativesbis}
A_2(D)\!\left( f^{(3)}(x) + \frac{6}{x} f^{(2)}(x) \right)
\quad \text{is odd.}
\end{equation}

\medskip\noindent
On the other hand, because $f = S_{0,1} [v_1]\,$, a second application of Lemma~\ref{ODESl} (1), now with $\kappa=0$, yields
\begin{equation}\label{eq:def-fbis}
f'(x) = \frac{2}{x}\, v_1(x) + v_1'(x)\,.
\end{equation}
Setting
\[
G(x):=4A_2(D)\!\left( f^{(3)}(x)+\frac{6}{x}f^{(2)}(x)\right)\,,
\]
a straightforward simplification yields the following differential
expression:
\[
\begin{aligned}
G(x)
=\;&
v_1^{(5)}(x)
+\Bigl(\frac{8}{x}-6\Bigr)v_1^{(4)}(x)
+\Bigl(-\frac{8}{x^{2}}-\frac{48}{x}+12\Bigr)v_1^{(3)}(x)
\\[0.5em]
&\quad
+\Bigl(-\frac{24}{x^{3}}+\frac{96}{x}\Bigr)v_1''(x)
+\Bigl(\frac{96}{x^{4}}+\frac{144}{x^{3}}+\frac{96}{x^{2}}\Bigr)v_1'(x)
\\[0.5em]
&\quad
+\Bigl(-\frac{96}{x^{5}}-\frac{144}{x^{4}}-\frac{96}{x^{3}}\Bigr)v_1(x)\,.
\end{aligned}
\]
We therefore recover exactly the same odd function \(G(x)\) as in the previous
case with \(v_1\) even. 

\vspace{0.2cm}\noindent
Writing $G(x)+G(1-x)=0$, we get :
\begin{equation}\label{EDO02}
\begin{aligned}
&\quad \phantom{+} \ \Biggl(
\frac{8}{x}+\frac{8}{1-x}-12
\Biggr) v_1^{(4)}(x)
\\[0.6em]
&\quad
+\Biggl(
-\frac{8}{x^{2}}+\frac{8}{(1-x)^{2}}
-\frac{48}{x}+\frac{48}{1-x}
\Biggr) v_1^{(3)}(x)
\\[0.6em]
&\quad
+\Biggl(
-\frac{24}{x^{3}}-\frac{24}{(1-x)^{3}}
+\frac{96}{x}+\frac{96}{1-x}
\Biggr) v_1''(x)
\\[0.6em]
&\quad
+\Biggl(
\frac{96}{x^{4}}-\frac{96}{(1-x)^{4}}
+\frac{144}{x^{3}}-\frac{144}{(1-x)^{3}}
+\frac{96}{x^{2}}-\frac{96}{(1-x)^{2}}
\Biggr) v_1'(x)
\\[0.6em]
&\quad
+\Biggl(
-\frac{96}{x^{5}}-\frac{96}{(1-x)^{5}}
-\frac{144}{x^{4}}-\frac{144}{(1-x)^{4}}
-\frac{96}{x^{3}}-\frac{96}{(1-x)^{3}}
\Biggr) v_1(x)
= 0 \, .
\end{aligned}
\end{equation}

\paragraph{Indicial roots and determination of the solution.}
Using \textsc{Mathematica}, we compute the indicial equation of \eqref{EDO02} at
the singular point $x=0$. This yields
\[
8(\rho-4)(\rho-3)(\rho-1)(\rho+1)=0\,,
\]
so that the indicial roots are
\[
\rho\in\{-1,\,1,\,3,\,4\}\,.
\]
We now look for the solution of \eqref{EDO02} satisfying the normalization
condition
\[
v_1\!\left(\tfrac12\right)=1\,.
\]
Since $v_1$ is even with respect to $x=\tfrac12\,$, we have
\[
v_1'\!\left(\tfrac12\right)=v_1^{(3)}\!\left(\tfrac12\right)=0\,,
\]
and, from \eqref{eq:DA2w1-midpoint}, it follows that
\[
v_1''\!\left(\tfrac12\right)=\frac{128}{5}\,.
\]
By the Cauchy--Lipschitz theorem, these conditions uniquely determine a solution
on $(0,1)$. Using \textsc{Mathematica}, we obtain the following explicit expression:
\begin{equation}\label{eq:v1-pfd}
\begin{aligned}
v_1(x)
={}&
\frac{5}{3}
-\frac{5}{8x}
-\frac{5}{8(1-x)}
-\frac{15}{8\,(2-3x+3x^2)}
+\frac{25}{12\,(2-3x+3x^2)^2}\,.
\end{aligned}
\end{equation}
This solution
exhibits a non-integrable blow--up at the boundary. In particular,
\[
v_1\notin L^2(0,1)\,.
\]
We deduce that one must impose
\[
v_1\!\left(\tfrac12\right)=0\,.
\]
It then follows that all derivatives of $v_1$ at $x=\tfrac12$ up to order three
vanish. By uniqueness of the Cauchy problem, this implies that
\[
v_1 \equiv 0 \qquad \text{on } (0,1)\,.
\]

\medskip\noindent
Thus, we have established the following injectivity result in the case $(0,2)$:

\begin{thm}[Injectivity for the pair $(0,2)$]
	\label{thm:injective-01b}
	For $(\kappa_1,\kappa_2)=(0,2)$\,, the Fr\'echet differential of the spectral map
	\[
	D_{(p,q)}\mathcal S_{0,2}(0,0)\,
	:\,
	L^2(0,1)\times L^2(0,1)
	\longrightarrow
	\ell^2_{\mathbb R}(\mathbb Z)\times \ell^2_{\mathbb R}(\mathbb Z)
	\]
	is injective.
\end{thm}



\subsection{Injectivity of the differential for the pair $(\kappa_1, \kappa_2)=(1,2)$}

Throughout this subsection, we assume that $(v_1,v_2) \in L^2(0,1)^2$ satisfies the linearized spectral condition
\[
D_{(p,q)} \lambda_{\kappa,n}(0,0)\cdot (v_1,v_2)=0\,,
\qquad n\in\mathbb{Z},
\]
for both effective angular momenta $\kappa=1$ and $\kappa=2\,$.

\medskip\noindent
Applying Theorem~\ref{mainthm} with $\kappa=1$ and $\kappa=2$, we obtain that
$A_\kappa(D)\bigl[T_\kappa^j[v_j]\bigr]$ is even for $j=1$ and odd for $j=2$,
with respect to $x=\tfrac12\,$.

\medskip\noindent
We first consider the  case $j=2$ with $\kappa=1$.
Set $f=S_{0,2}[v_2]$. A direct computation gives
\begin{equation}
	2A_1(D)\bigl[T_1^2[v_2]\bigr](x)
	=
	2A_1(D)\bigl[S_{0,2}[v_2]\bigr](x)
	=
	2A_1(D)[f](x)
	=
	-f'(x)+2f(x)\,.
\end{equation}
Decomposing $f$ into its even and odd parts,
\begin{equation}\label{decomparity}
	f=f_e+f_o\,,
\end{equation}
where $f_e$ is even and $f_o$ is odd with respect to $x=\tfrac12\,$,
we immediately obtain
\begin{equation}\label{linkparity}
	f_e=\frac12\,f_o'\,,
\end{equation}
since the function $- f_o' + 2 f_e$ is even and odd and therefore
\begin{equation}\label{formf}
	f=\frac12\,f_o'+f_o\,.
\end{equation}
\medskip
\noindent
We now exploit the case $\kappa=2\,$. We define
\[
g := S_{1,2}[f]
= f(x)-6x^{3}\int_x^1 \frac{f(t)}{t^{4}}\,dt,
\]
so that $g=-T_2^2[v_2]$. By assumption,
$
A_2(D)\bigl[T_2^2[v_2]\bigr]
$
is odd in the sense of distributions, where 
\[
A_2(t)=\tfrac14 t^2-\tfrac32 t+3\,.
\]
A straightforward computation yields
\begin{equation}\label{parity1}
	4A_2(D)[g](x)
	=
	f''(x)
	+6\Bigl(\frac{1}{x}-1\Bigr)f'(x)
	+\Bigl(\frac{12}{x^2}-\frac{36}{x}+12\Bigr)f(x)
	-36x(1-x)(1-2x)\int_x^1 \frac{f(t)}{t^4}\,dt\,.
\end{equation}
Since $4A_2(D)[g]$ is odd with respect to $x=\tfrac12\,$, evaluating \eqref{parity1} at $x=\tfrac12$ yields
\[
f''\!\left(\tfrac12\right)
+6f'\!\left(\tfrac12\right)
-12f\!\left(\tfrac12\right)
=0\,.
\]
Using the decomposition $f=\tfrac12 f_o'+f_o$\,, where $f_o$ is odd, we obtain
\begin{equation}\label{eq:fo-third-derivative}
	f_o^{(3)} \!\left(\tfrac12\right)
	=0\,.
\end{equation}
Similarly, differentiating \eqref{parity1} twice yields
\begin{equation}\label{eq:D2A2g}
	\begin{aligned}
		D^{2}\bigl(4A_2(D)[g]\bigr)(x)
		&=
		f^{(4)}(x)
		+6\Bigl(\frac{1}{x}-1\Bigr)f^{(3)}(x)
		+\Bigl(12-\frac{36}{x}\Bigr)f''(x)
		\\[0.3em]
		&\quad
		+\Bigl(\frac{72}{x}-\frac{36}{x^{2}}\Bigr)f'(x)
		+\Bigl(\frac{144}{x^{2}}-\frac{72}{x^{3}}\Bigr)f(x)
		\\[0.3em]
		&\quad
		+216(1-2x)\int_x^1 \frac{f(t)}{t^{4}}\,dt\,.
	\end{aligned}
\end{equation}
Evaluating the identity \eqref{eq:D2A2g} at $x=\tfrac12$ yields
\begin{equation}\label{eq:fo-fifth-derivative}
	f_o^{(5)}\!\left(\tfrac12\right)=48\,f_o^{(3)}\!\left(\tfrac12\right)=0\,.
\end{equation}
Finally, a similar computation yields
\begin{equation}\label{eq:D4A2g}
	\begin{aligned}
		G(x):= D^{4}\bigl(4A_2(D)[g]\bigr)(x)
		&=
		f^{(6)}(x)
		+6\Bigl(\frac{1}{x}-1\Bigr)f^{(5)}(x)
		\\[0.3em]
		&\quad
		+\Bigl(12-\frac{36}{x}-\frac{12}{x^{2}}\Bigr)f^{(4)}(x)
		+\Bigl(\frac{72}{x}+\frac{36}{x^{2}}+\frac{12}{x^{3}}\Bigr)f^{(3)}(x)\,.
	\end{aligned}
\end{equation}
Replacing $f=\frac12 f_o'+f_o\,$, we get
\begin{equation}\label{eq:G-fo-expanded}
	\begin{aligned}
		G(x)
		&=
		\frac12\,f_o^{(7)}(x)
		+\Bigl(\frac{3}{x}-2\Bigr)f_o^{(6)}(x)
		\\[0.3em]
		&\quad
		-\Bigl(\frac{12}{x}+\frac{6}{x^{2}}\Bigr)f_o^{(5)}(x)
		+\Bigl(12+\frac{6}{x^{2}}+\frac{6}{x^{3}}\Bigr)f_o^{(4)}(x)
		\\[0.3em]
		&\quad
		+\Bigl(\frac{72}{x}+\frac{36}{x^{2}}+\frac{12}{x^{3}}\Bigr)f_o^{(3)}(x)\,.
	\end{aligned}
\end{equation}
We now use the symmetry condition $G(x)+G(1-x)=0$. 
This yields the following differential equation:
\begin{equation}\label{eq:symmetry-fo}
	\begin{aligned}
		0
		&=
		f_o^{(7)}(x)
		+3\Bigl(\frac{1}{x}-\frac{1}{1-x}\Bigr)f_o^{(6)}(x)
		\\[0.3em]
		&\quad
		-\Bigl(\frac{12}{x}+\frac{6}{x^{2}}+\frac{12}{1-x}+\frac{6}{(1-x)^{2}}\Bigr)f_o^{(5)}(x)
		\\[0.3em]
		&\quad
		+6\Bigl(\frac{1}{x^{2}}+\frac{1}{x^{3}}-\frac{1}{(1-x)^{2}}-\frac{1}{(1-x)^{3}}\Bigr)f_o^{(4)}(x)
		\\[0.3em]
		&\quad
		+\Bigl(\frac{72}{x}+\frac{36}{x^{2}}+\frac{12}{x^{3}}
		+\frac{72}{1-x}+\frac{36}{(1-x)^{2}}+\frac{12}{(1-x)^{3}}\Bigr)f_o^{(3)}(x)\,.
	\end{aligned}
\end{equation}
Finally, introducing
\[
y(x)=f_o^{(3)}(x)\,,
\]
we are led to a fourth--order differential equation satisfied by $y$. 
The function $y$ is even with respect to $x=\tfrac12\,$, and the previous identities imply
\[
y^{(k)}\!\left(\tfrac12\right)=0, \qquad k=0,1,2,3\,.
\]
The Cauchy--Lipschitz theorem then yields that
$y\equiv 0$. Consequently, $f_o$ must be an odd polynomial of degree at most two,
hence it necessarily takes the form
\begin{equation}\label{eq:fo-linear}
	f_o(x)=a\,(x-\tfrac12)\,.
\end{equation}
Using once more the relation $f=\tfrac12 f_o'+f_o\,$, we deduce
\begin{equation}\label{eq:f-linear}
	f(x)=\frac{a}{2}+a\,(x-\tfrac12)=a\,x\,.
\end{equation}
Since $f=S_{0,2}[v_2]\,$, we recover $v_2$ by applying the left inverse
$S_{0,1}^*$ given in Lemma~\ref{akns_lemme_reduction_indice}(v), namely
\[
v_2(x)=S_{0,1}^*[f](x)=f(x)-\frac{2}{x}\int_0^x f(t)\,dt\,.
\]
A direct computation yields
\begin{equation}\label{eq:v2-explicit}
	v_2(x)
	=0\,.
\end{equation}

\vspace{0.5cm}\noindent
Let us now examine the case $j=1\,$. 
By Theorem~\ref{mainthm} applied with $\kappa=1$ and $\kappa=2$, we know that
$
A_\kappa(D)\bigl[T_\kappa^1[v_1]\bigr]
$
is even. Set $f=S_{0,1}[v_1]\,$. As in the case $j=2\,$, the analysis of the case $\kappa=1$ yields
\begin{equation}\label{formfbis}
	f=f_e + \frac12\,f_e'\,.
\end{equation}
(As before, $f_o$ denotes the odd part of $f$ with respect to $x=\tfrac12\,$, and $f_e$ its even part.)
\medskip
\noindent
We now exploit the case $\kappa=2$. We define
\[
g := S_{1,1}[f]
= f(x)-6x^{2}\int_x^1 \frac{f(t)}{t^{3}}\,dt\,,
\]
so that $g=-T_2^1[v_1]\,$. By assumption,
$
D A_2(D)\bigl[T_2^1[v_1]\bigr]
$
is odd in the sense of distributions. A straightforward computation yields
\begin{equation}\label{parity1bis}
	\begin{aligned}
		4D A_2(D)[g](x)
		={}& f^{(3)}(x)
		+6\Bigl(\frac{1}{x}-1\Bigr)f''(x)
		+\Bigl(12-\frac{36}{x}\Bigr)f'(x)
		+\Bigl(\frac{72}{x}-\frac{36}{x^{2}}\Bigr)f(x)\\
		&\quad
		+72(1-2x)\int_x^1 \frac{f(t)}{t^{3}}\,dt\,.
	\end{aligned}
\end{equation}
Since $4D A_2(D)[g]$ is odd with respect to $x=\tfrac12\,$, evaluating
\eqref{parity1bis} at $x=\tfrac12$ gives
\begin{equation}\label{eq:fe-second-derivative-half}
	f_o^{(3)}\!\left(\tfrac12\right)=48 f_o'\!\left(\tfrac12\right)\, .
\end{equation}
Similarly, differentiating \eqref{parity1bis} twice yields
\begin{equation}
	\begin{aligned}
		H(x):=4D^{3}A_2(D)[g](x)
		={}& f^{(5)}(x)
		+6\Bigl(\frac{1}{x}-1\Bigr)f^{(4)}(x)\\
		&\quad
		+\Bigl(12-\frac{36}{x}-\frac{12}{x^{2}}\Bigr)f^{(3)}(x)\\
		&\quad
		+\Bigl(\frac{72}{x}+\frac{36}{x^{2}}+\frac{12}{x^{3}}\Bigr)f''(x)\,.
	\end{aligned}
\end{equation}
Replacing $f=f_e+\tfrac12 f_e'$ in $H(x)$ and using the relation
$f_o=\tfrac12 f_e'\,$, we can express $H$ entirely in terms of $f_o\,$.
A straightforward computation yields
\begin{equation}\label{eq:H-fo-only}
	\begin{aligned}
		H(x)
		={}& f_o^{(5)}(x)
		+\Bigl(\frac{6}{x}-4\Bigr) f_o^{(4)}(x)
		-\Bigl(\frac{24}{x}+\frac{12}{x^{2}}\Bigr) f_o^{(3)}(x)\\
		&\quad
		+\Bigl(24+\frac{12}{x^{2}}+\frac{12}{x^{3}}\Bigr) f_o''(x)
		+\Bigl(\frac{144}{x}+\frac{72}{x^{2}}+\frac{24}{x^{3}}\Bigr) f_o'(x)\,.
	\end{aligned}
\end{equation}
We now use the symmetry condition $H(x)+H(1-x)=0$. 
This yields the following differential equation:
\begin{equation}\label{edo-fo-final}
	\begin{aligned}
		0={}&
		2f_o^{(5)}(x)
		+\Bigl(\frac{6}{x}-\frac{6}{1-x}\Bigr)f_o^{(4)}(x)\\
		&-\Bigl(\frac{24}{x}+\frac{24}{1-x}
		+\frac{12}{x^{2}}+\frac{12}{(1-x)^{2}}\Bigr)f_o^{(3)}(x)\\
		&+\Bigl(\frac{12}{x^{2}}+\frac{12}{x^{3}}
		-\frac{12}{(1-x)^{2}}-\frac{12}{(1-x)^{3}}\Bigr)f_o''(x)\\
		&+\Bigl(\frac{144}{x}+\frac{144}{1-x}
		+\frac{72}{x^{2}}+\frac{72}{(1-x)^{2}}
		+\frac{24}{x^{3}}+\frac{24}{(1-x)^{3}}\Bigr)f_o'(x)\,.
	\end{aligned}
\end{equation}
Finally, introducing
$
y(x)=f_o'(x)\,,
$
we are led to a fourth--order differential equation satisfied by $y$:
\begin{equation}\label{edo-y}
	\begin{aligned}
		0={}&
		2y^{(4)}(x)
		+\Bigl(\frac{6}{x}-\frac{6}{1-x}\Bigr)y^{(3)}(x)\\
		&-\Bigl(\frac{24}{x}+\frac{24}{1-x}
		+\frac{12}{x^{2}}+\frac{12}{(1-x)^{2}}\Bigr)y''(x)\\
		&+\Bigl(\frac{12}{x^{2}}+\frac{12}{x^{3}}
		-\frac{12}{(1-x)^{2}}-\frac{12}{(1-x)^{3}}\Bigr)y'(x)\\
		&+\Bigl(\frac{144}{x}+\frac{144}{1-x}
		+\frac{72}{x^{2}}+\frac{72}{(1-x)^{2}}
		+\frac{24}{x^{3}}+\frac{24}{(1-x)^{3}}\Bigr)y(x)\,.
	\end{aligned}
\end{equation}
A direct computation shows that the roots of the indicial equation are $-2$, $0$, $2$, and $3\,$.

\medskip\noindent
The function $y$ is even with respect to $x=\tfrac12$. 
From the previous computations, if we normalize by choosing 
$f_o'\!\left(\tfrac12\right)=1\,$, then
\[
y\!\left(\tfrac12\right)=1\,, 
\qquad
y'\!\left(\tfrac12\right)=0\,,
\qquad
y''\!\left(\tfrac12\right)=48\,.
\qquad
y^{(3)}\!\left(\tfrac12\right)=0\,.
\]
Using  \textsc{Mathematica}, the unique solution is given by
\begin{equation}\label{solution-y}
	y(x)= -1 + \frac{1}{4(x-1)^2} + \frac{1}{4x^2}\,.
\end{equation}
We recall that $y = f_0'$. Since $f_0$ is odd, we obtain
\[ f_0 (x) = -x + \frac{1}{4(1-x)} - \frac{1}{4x} + \frac{1}{2}\,,\]
and, since $f_e' = 2f_0\,$, there exists a real constant $C$ such that
\[f_e(x) = 2 \left( -\frac{x^2}{2} - \frac{1}{4} \ln(1-x) - \frac{1}{4} \ln(x)  + \frac{x}{2} \right) + C\,.\]
We thus obtain,
\begin{equation}\label{solution-f}
	f(x) = f_e(x) + f_0(x) = -x^2 - \frac{1}{2} \ln(1-x) - \frac{1}{2} \ln(x)  + \frac{1}{4(1-x)} - \frac{1}{4x} + C\, .
\end{equation}
This leads to a contradiction, since \(f=S_{0,1}[v_1]\) must belong to
\(L^2(0,1)\), whereas the function \eqref{solution-f} is not square integrable
near \(x=0\) and \(x=1\,\).
Consequently, the initial condition must satisfy
\[
f_o'\!\left(\tfrac12\right)=0\,.
\]
By the Cauchy--Lipschitz Theorem, the corresponding solution of the differential
equation then satisfies \(y\equiv 0 \,\).
Hence \(f_o\) is a constant, which must be zero since \(f_o\) is odd.
Therefore \(f\) itself is a constant function.

\vspace{0.2cm}\noindent
Since \(f=S_{0,1}[v_1]\,\), we recover \(v_1\) by applying the left inverse
\(S_{0,2}^*\) given in Lemma~\ref{akns_lemme_reduction_indice}(v), namely
\[
v_1(x)=S_{0,2}^*[f](x)
= f(x)-\frac{2}{x^{2}}\int_0^x t\,f(t)\,dt\, .
\]
A direct computation yields
\begin{equation}\label{eq:v1-explicitb}
	v_1(x)=0\,.
\end{equation}

\medskip\noindent
Thus, we have established the following injectivity result in the case $(1,2)$:

\begin{thm}[Injectivity for the pair $(1,2)$]
	\label{thm:injective-01c}
	For $(\kappa_1,\kappa_2)=(1,2)\,$, the Fr\'echet differential of the spectral map
	\[
	D_{(p,q)}\mathcal S_{1,2}(0,0)
	:
	L^2(0,1)\times L^2(0,1)
	\longrightarrow
	\ell^2_{\mathbb R}(\mathbb Z)\times \ell^2_{\mathbb R}(\mathbb Z)
	\]
	is one to one.
\end{thm}



\subsection{Injectivity of the differential for the pair $(\kappa_1, \kappa_2)=(0,3)$}

Throughout this subsection, we consider  $(v_1,v_2) \in L^2(0,1)^2$ satisfying the
linearized spectral constraint
\[
D_{(p,q)}\lambda_{\kappa,n}(0,0)\cdot (v_1,v_2)=0\,,
\qquad n\in\mathbb{Z},
\]
simultaneously for the two effective angular momenta $\kappa=0$ and $\kappa=3$.

\medskip\noindent
In the case $\kappa=0$, one has, as in the previous section, that $v_1$ is even and
$v_2$ is odd with respect to the midpoint $x=\tfrac12\,$.

\medskip\noindent
We now turn to the case $\kappa = 3\,$.
According to Theorem~\ref{mainthm}, and setting $w_j := T_3^{\,j}[v_j]$,
$j=1,2$, we have
\begin{equation}\label{eq:schrodinger_condition_l3}
	A_3(D)\, \bigl[w_j\bigr]
	\quad \text{is even if $j=1$, and odd if $j=2\,$,}
\end{equation}
where
\[
A_3(t) = -\tfrac{1}{8}t^3 + \tfrac{3}{2}t^2 - \tfrac{15}{2}t + 15\,,
\qquad D=\tfrac{d}{dx}\,.
\]

\vspace{0.3cm}\noindent
We first analyze the case $j=2$. The relevant transformation operator can be
written explicitly. For $x\in(0,1)$, one has
\begin{equation}\label{eq:T3-formulas}
	\begin{aligned}
		T_3^2[v_2](x)
		&=
		v_2(x)
		-6x\int_x^1\frac{v_2(t)}{t^2}\,dt
		+48x^3\int_x^1\frac{v_2(t)}{t^4}\,dt
		-60x^5\int_x^1\frac{v_2(t)}{t^6}\,dt\,.
	\end{aligned}
\end{equation}
Moreover, introducing the differential operator $D=\frac{d}{dx}$ and the polynomial
\[
A_3(t)=-\tfrac18 t^3+\tfrac32 t^2-\tfrac{15}{2}t+15\,,
\qquad\text{so that}\qquad
8A_3(D)=-D^3+12D^2-60D+120\,,
\]
we obtain the following explicit identity

\begin{equation}\label{eq:8A3D-T32}
	\begin{aligned}
		F(x) &= 8A_3(D)\bigl[T_3^2[v_2]\bigr](x)
		\\
		&= -\,v_2^{(3)}(x)
		+\Bigl(12-\frac{18}{x}\Bigr)v_2''(x)
		+\Bigl(-60+\frac{216}{x}-\frac{126}{x^{2}}\Bigr)v_2'(x)
		\\
		&\quad
		+\Bigl(120-\frac{1080}{x}+\frac{1728}{x^{2}}-\frac{624}{x^{3}}\Bigr)v_2(x)
		+360(1-2x)\int_x^1\frac{v_2(t)}{t^{2}}\,dt
		\\
		&\quad
		+288\bigl(20x^{3}-30x^{2}+12x-1\bigr)\int_x^1\frac{v_2(t)}{t^{4}}\,dt
		\\
		&\quad
		-3600x^{2}\bigl(2x^{3}-5x^{2}+4x-1\bigr)\int_x^1\frac{v_2(t)}{t^{6}}\,dt\,.
	\end{aligned}
\end{equation}
Evaluating \eqref{eq:8A3D-T32} at $x=\tfrac12\,$, we use that the left-hand side
is odd with respect to $\tfrac12$, hence it vanishes at $x=\tfrac12$.
Since $v_2$ is also odd, one has
\[
v_2\!\left(\tfrac12\right)=0\,,
\qquad
v_2''\!\left(\tfrac12\right)=0\,,
\]
Thus,
\[
v_2^{(3)}\!\left(\tfrac12\right)=-132\,v_2'\!\left(\tfrac12\right)\,.
\]

\vspace{0.2cm}\noindent
We now compute the second derivative of \(F(x)=8A_3(D)\bigl[T_3^2[v_2]\bigr](x)\)\,.
Differentiating~\eqref{eq:8A3D-T32} twice, we obtain
\begin{equation}\label{eq:Fpp}
	\begin{aligned}
		F''(x)
		&=
		-\,v_2^{(5)}(x)
		+\Bigl(12-\frac{18}{x}\Bigr)v_2^{(4)}(x)
		+\Bigl(-60+\frac{216}{x}-\frac{90}{x^{2}}\Bigr)v_2^{(3)}(x)
		\\
		&\quad
		+\Bigl(120-\frac{1080}{x}+\frac{1296}{x^{2}}-\frac{156}{x^{3}}\Bigr)v_2''(x)
		\\
		&\quad
		+\frac{36\bigl(60x^{3}-210x^{2}+124x-9\bigr)}{x^{4}}\,v_2'(x)
		+\frac{1440\bigl(12x^{3}-26x^{2}+12x-1\bigr)}{x^{5}}\,v_2(x)
		\\
		&\quad
		+17280(2x-1)\int_x^1\frac{v_2(t)}{t^{4}}\,dt
		-7200\bigl(20x^{3}-30x^{2}+12x-1\bigr)\int_x^1\frac{v_2(t)}{t^{6}}\,dt\,.
	\end{aligned}
\end{equation}
Evaluating~\eqref{eq:Fpp} at \(x=\tfrac12\) and using the oddness of \(F\) and \(v_2\),
all nonlocal terms vanish and we obtain
\[
0
=
4608\,v_2'\!\left(\tfrac12\right)
+12\,v_2^{(3)}\!\left(\tfrac12\right)
- v_2^{(5)}\!\left(\tfrac12\right)\,.
\]

\vspace{0.2cm}\noindent
We now differentiate~\eqref{eq:Fpp} twice. This yields
\begin{equation}\label{eq:F4}
	\begin{aligned}
		F^{(4)}(x)
		&=
		-\,v_2^{(7)}(x)
		+\Bigl(12-\frac{18}{x}\Bigr)v_2^{(6)}(x)
		+\Bigl(-60+\frac{216}{x}-\frac{54}{x^{2}}\Bigr)v_2^{(5)}(x)
		\\
		&\quad
		+\Bigl(120-\frac{1080}{x}+\frac{864}{x^{2}}+\frac{168}{x^{3}}\Bigr)v_2^{(4)}(x)
		\\
		&\quad
		+\Bigl(\frac{2160}{x}-\frac{5400}{x^{2}}-\frac{288}{x^{3}}+\frac{72}{x^{4}}\Bigr)v_2^{(3)}(x)
		\\
		&\quad
		+\Bigl(\frac{12960}{x^{2}}-\frac{9360}{x^{3}}-\frac{1728}{x^{4}}-\frac{720}{x^{5}}\Bigr)v_2''(x)
		\\
		&\quad
		+\Bigl(\frac{44640}{x^{3}}-\frac{19440}{x^{4}}+\frac{1728}{x^{5}}+\frac{720}{x^{6}}\Bigr)v_2'(x)
		+\Bigl(\frac{172800}{x^{4}}-\frac{86400}{x^{5}}\Bigr)v_2(x)
		\\
		&\quad
		+432000\,(1-2x)\int_x^1\frac{v_2(t)}{t^{6}}\,dt\,.
	\end{aligned}
\end{equation}
Evaluating at \(x=\tfrac12\,\), since \(F\) is odd with respect to \(x=\tfrac12\) one has
\(F^{(4)}(\tfrac12)=0\), and the nonlocal term vanishes. Moreover, since \(v_2\) is
also odd with respect to \(x=\tfrac12\), we have
\(v_2(\tfrac12)=v_2''(\tfrac12)=v_2^{(4)}(\tfrac12)=v_2^{(6)}(\tfrac12)=0\)\,.
Therefore,
\[
0
=
-\,v_2^{(7)}\!\left(\tfrac12\right)
+156\,v_2^{(5)}\!\left(\tfrac12\right)
-18432\,v_2^{(3)}\!\left(\tfrac12\right)
+147456\,v_2'\!\left(\tfrac12\right)\,.
\]

\vspace{0.2cm}\noindent
Differentiating twice~\eqref{eq:F4} and collecting terms, we obtain
\begin{equation}\label{eq:F6}
	\begin{aligned}
		F^{(6)}(x)
		&=
		-\,v_2^{(9)}(x)
		+\Bigl(12-\frac{18}{x}\Bigr)v_2^{(8)}(x)
		+\Bigl(-60+\frac{216}{x}-\frac{18}{x^{2}}\Bigr)v_2^{(7)}(x)
		\\
		&\quad
		+\Bigl(120-\frac{1080}{x}+\frac{432}{x^{2}}+\frac{348}{x^{3}}\Bigr)v_2^{(6)}(x)
		\\
		&\quad
		+\Bigl(\frac{2160}{x}-\frac{3240}{x^{2}}-\frac{3312}{x^{3}}-\frac{1260}{x^{4}}\Bigr)v_2^{(5)}(x)
		\\
		&\quad
		+\Bigl(\frac{8640}{x^{2}}+\frac{10080}{x^{3}}+\frac{5184}{x^{4}}+\frac{720}{x^{5}}\Bigr)v_2^{(4)}(x)
		\\
		&\quad
		+\Bigl(-\frac{2880}{x^{3}}+\frac{4320}{x^{4}}+\frac{12096}{x^{5}}+\frac{9360}{x^{6}}\Bigr)v_2^{(3)}(x)
		\\
		&\quad
		+\Bigl(-\frac{17280}{x^{4}}-\frac{43200}{x^{5}}-\frac{51840}{x^{6}}-\frac{30240}{x^{7}}\Bigr)v_2''(x)
		\\
		&\quad
		+\Bigl(\frac{17280}{x^{5}}+\frac{43200}{x^{6}}+\frac{51840}{x^{7}}+\frac{30240}{x^{8}}\Bigr)v_2'(x)\,.
	\end{aligned}
\end{equation}

\vspace{0.2cm}\noindent
Let \(G(x)=F^{(6)}(x)\,\). Since \(F\) is odd with respect to \(x=\tfrac12\), the same holds for \(G\),
and therefore
\[
G(x)+G(1-x)=0\,,\qquad x\in(0,1)\,.
\]
Using that \(v_2\) is also odd  and applying the reduction obtained above,
we can rewrite the symmetry identity \(G(x)+G(1-x)=0\) as a linear ODE for
\[
w:=v_2'.
\]
This yields the following eighth--order symmetry equation:
\begin{equation}\label{EDO-w}
	\begin{aligned}
		&\quad -2\,w^{(8)}(x)
		\\[0.6em]
		&\quad
		-\Biggl(
		\frac{18}{x}-\frac{18}{1-x}
		\Biggr) w^{(7)}(x)
		\\[0.6em]
		&\quad
		+\Biggl(
		-120
		+216\Bigl(\frac{1}{x}+\frac{1}{1-x}\Bigr)
		-18\Bigl(\frac{1}{x^{2}}+\frac{1}{(1-x)^{2}}\Bigr)
		\Biggr) w^{(6)}(x)
		\\[0.6em]
		&\quad
		-\Biggl(
		\frac{1080}{x}-\frac{1080}{1-x}
		-\frac{432}{x^{2}}+\frac{432}{(1-x)^{2}}
		-\frac{348}{x^{3}}+\frac{348}{(1-x)^{3}}
		\Biggr) w^{(5)}(x)
		\\[0.6em]
		&\quad
		+\Biggl(
		\frac{2160}{x}+\frac{2160}{1-x}
		-\frac{3240}{x^{2}}-\frac{3240}{(1-x)^{2}}
		-\frac{3312}{x^{3}}-\frac{3312}{(1-x)^{3}}
		-\frac{1260}{x^{4}}-\frac{1260}{(1-x)^{4}}
		\Biggr) w^{(4)}(x)
		\\[0.6em]
		&\quad
		+\Biggl(
		\frac{8640}{x^{2}}-\frac{8640}{(1-x)^{2}}
		+\frac{10080}{x^{3}}-\frac{10080}{(1-x)^{3}}
		+\frac{5184}{x^{4}}-\frac{5184}{(1-x)^{4}}
		+\frac{720}{x^{5}}-\frac{720}{(1-x)^{5}}
		\Biggr) w^{(3)}(x)
		\\[0.6em]
		&\quad
		+\Biggl(
		-\frac{2880}{x^{3}}-\frac{2880}{(1-x)^{3}}
		+\frac{4320}{x^{4}}+\frac{4320}{(1-x)^{4}}
		+\frac{12096}{x^{5}}+\frac{12096}{(1-x)^{5}}
		+\frac{9360}{x^{6}}+\frac{9360}{(1-x)^{6}}
		\Biggr) w''(x)
		\\[0.6em]
		&\quad
		+\Biggl(
		-\frac{17280}{x^{4}}+\frac{17280}{(1-x)^{4}}
		-\frac{43200}{x^{5}}+\frac{43200}{(1-x)^{5}}
		-\frac{51840}{x^{6}}+\frac{51840}{(1-x)^{6}}
		-\frac{30240}{x^{7}}+\frac{30240}{(1-x)^{7}}
		\Biggr) w'(x)
		\\[0.6em]
		&\quad
		+\Biggl(
		\frac{17280}{x^{5}}+\frac{17280}{(1-x)^{5}}
		+\frac{43200}{x^{6}}+\frac{43200}{(1-x)^{6}}
		+\frac{51840}{x^{7}}+\frac{51840}{(1-x)^{7}}
		+\frac{30240}{x^{8}}+\frac{30240}{(1-x)^{8}}
		\Biggr) w(x)
		=0\,.
	\end{aligned}
\end{equation}

\vspace{0.2cm}\noindent
The  indicial roots at $x=0$ are
\[
r\in\{7,6,5,3,1,-1\}\ \cup\ \{-1\pm i\sqrt{23}\}\,.
\]

\medskip
\noindent
From the previous analysis, solutions of this ODE are uniquely determined by the
single parameter $v_2'\!\left(\tfrac12\right)$. For instance, imposing
the normalization
\[
w\!\left(\tfrac12\right)= v_2'\!\left(\tfrac12\right)=1
\]
gives us a solution $w=v_2'$ of \eqref{EDO-w} satisfying  the differential constraints at $x=\tfrac12\,$,
\[
v_2^{(3)}\!\left(\tfrac12\right)=-132\,,\qquad
v_2^{(5)}\!\left(\tfrac12\right)=3024\,,\qquad
v_2^{(7)}\!\left(\tfrac12\right)=3052224\,,
\]
which follow from the symmetry relations derived above. 

\medskip
\noindent
To gain further insight into the behaviour of solutions, we performed a numerical
integration of equation \eqref{EDO-w} using \textsc{Mathematica}.  Starting from
the normalization $v_2'\!\left(\tfrac12\right)=1$ together with the differential
constraints above, the resulting functions $w(x)=v_2'(x)$ and $v_2(x)$ are
displayed in Figure~\ref{fig:num-wv2}.

\begin{figure}[h]
	\centering
	\begin{minipage}{0.48\textwidth}
		\centering
		\includegraphics[width=\textwidth]{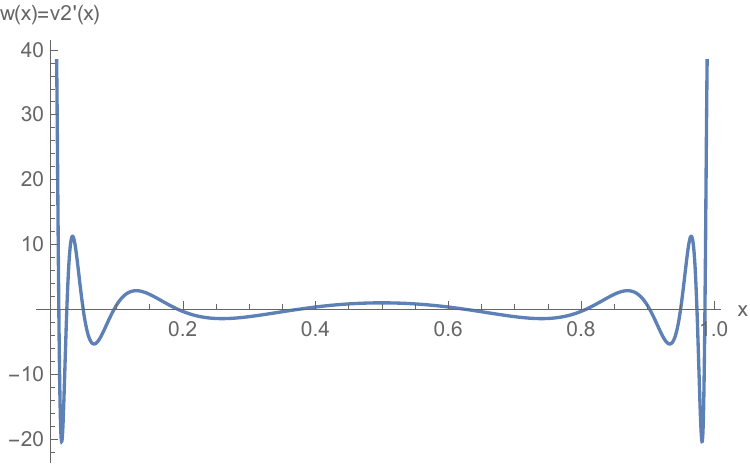}
	\end{minipage}
	\hfill
	\begin{minipage}{0.48\textwidth}
		\centering
		\includegraphics[width=\textwidth]{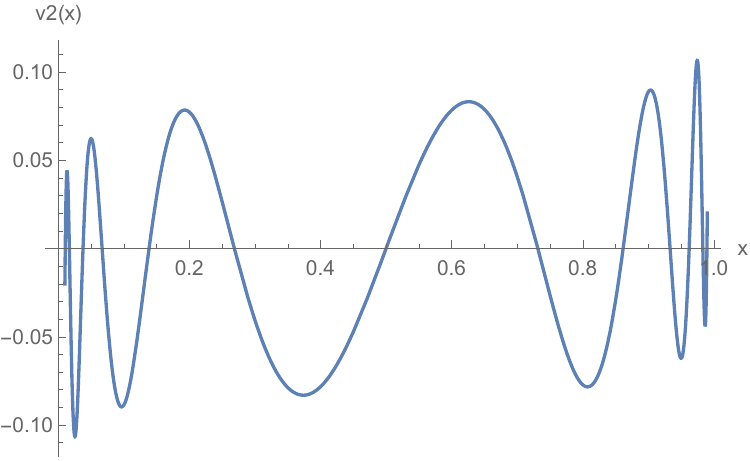}
	\end{minipage}
	\caption{Numerical solutions $w(x)=v_2'(x)$ (left) and $v_2(x)$ (right),
		computed with \textit{Mathematica}.}
	\label{fig:num-wv2}
\end{figure}

\noindent
The qualitative behaviour of $w$ is consistent with the structure of the
indicial roots. In particular, the complex pair $-1\pm i\sqrt{23}$ produces
oscillatory components in the local behaviour near the singular endpoints.

\medskip
\noindent
On the other hand, when $\kappa=3$, the previous analysis (see \eqref{eq:single-ell-kernel}) yields the additional
constraint
\[
\int_0^1 v_2(x)\,x^6\,dx=0\,.
\]
However, using the numerical solution corresponding to the normalization
$v_2'\!\left(\tfrac12\right)=1\,$, we obtain
\[
\int_0^1 v_2(x)\,x^6\,dx \approx 1.6\times10^{-1}\,,
\]
which is clearly non--zero. This leads to a numerical contradiction. Consequently, one must have
\[
v_2'\!\left(\tfrac12\right)=0\,.
\]
By the Cauchy--Lipschitz theorem applied to equation \eqref{EDO-w}, this
implies that $w\equiv0$. Since $v_2$ is odd with respect to $x=\tfrac12\,$, it
follows that $v_2\equiv0$\,.


\vspace{0.5cm}\noindent
We now turn to the case $j=1\,$. The analysis is completely analogous to the case
$j=2\,$. Since $v_1$ and $A_3(D)\bigl[T_3^1[v_1]\bigr]$ are even with respect to
$x=\tfrac12\,$, the successive symmetry identities at $x=\tfrac12$ determine all
higher even derivatives of $v_1$ from the two parameters
$v_1\!\left(\tfrac12\right)$ and $v_1''\!\left(\tfrac12\right)\,$, while all odd
derivatives vanish at $x=\tfrac12\,$. Namely,
\[
v_1^{(4)}\!\left(\tfrac12\right)=12\,v_1''\!\left(\tfrac12\right)+4608\,v_1\!\left(\tfrac12\right)\,.
\]
\[
v_1^{(6)}\!\left(\tfrac12\right)
=
156\,v_1^{(4)}\!\left(\tfrac12\right)
-18432\,v_1''\!\left(\tfrac12\right)
+147456\,v_1\!\left(\tfrac12\right)\,.
\]
Then, proceeding exactly as in the case $j=2$, one obtains the same eighth--order
symmetry equation as above, with $w$ replaced by $v_1$. Hence, once the two parameters
$v_1\!\left(\tfrac12\right)$ and $v_1''\!\left(\tfrac12\right)$ are fixed,
all higher derivatives at $x=\tfrac12$ are uniquely determined, and the
Cauchy--Lipschitz theorem yields a unique local even solution.

\medskip\noindent
We therefore introduce the two fundamental even solutions corresponding to the
initial data
\[
u:\quad \bigl(v_1(\tfrac12),v_1''(\tfrac12)\bigr)=(1,0)\,,
\qquad
v:\quad \bigl(v_1(\tfrac12),v_1''(\tfrac12)\bigr)=(0,1)\,.
\]
Any even solution of the symmetry equation is then a linear combination
\[
v_1=\alpha\,u+\beta\,v \,.
\]
To understand the behaviour of these solutions near $x=0$, we performed a
numerical Frobenius analysis using \textsc{Mathematica}. Starting from the
two independent even solutions $(u,v)$, we integrate the eighth--order equation
numerically on $(0,1)$\,.

\medskip\noindent
We have previously seen that any local solution near $x=0$ is expected to have
an expansion of the form
\[
v_1(x)
=
x^{-1}\!\left(
A + B\cos(\sqrt{23}\log x)
+ C\sin(\sqrt{23}\log x)
\right)
+ O(x)\,.
\]
For the pair $(u,v)$, \textsc{Mathematica} produces the numerical triples
\[
(A_u,B_u,C_u)=(0.700136,-0.0937512,-0.0416037)\,,
\]
\[
(A_v,B_v,C_v)=(0.00529329,0.00201729,-0.00117865)\,.
\]
To test whether a linear combination of these solutions could cancel the
leading singular behaviour, we solve numerically
\[
\alpha (A_u,B_u,C_u)+\beta (A_v,B_v,C_v)=(0,0,0)\,.
\]
The computation yields only the trivial solution $\alpha=\beta=0$.
Consequently, the only solution which is $L^2$ at both endpoints $x=0$ and $x=1$
is the trivial one. This numerical analysis leads to the following theorem.

\begin{thm}[Injectivity for the pair $(0,3)$]
	\label{thm:noninj-03}
	For $(\kappa_1,\kappa_2)=(0,3)\,$, the Fr\'echet differential of the spectral map
	\[
	D_{(p,q)}\mathcal S_{0,3}(0,0)\,
	:\,
	L^2(0,1)\times L^2(0,1)
	\longrightarrow
	\ell^2_{\mathbb R}(\mathbb Z)\times \ell^2_{\mathbb R}(\mathbb Z)
	\]
	is  one to one.
\end{thm}



\section{Closed range of the linearized spectral map}

In this section, we study the Fr\'echet differential of the spectral map at the
zero potential and prove that its range is closed when $\kappa_2-\kappa_1$ is odd.
For the corresponding radial Schr\"odinger problem, the closed--range property
was established by  Carlson-Shubin and Shubin-Christ see, e.g.,
\cite{CarlShu94, Sh}.

\subsection{Preliminaries}

\medskip\noindent
Let
\[
S:=D_{(p,q)}\mathcal S_{\kappa_1,\kappa_2}(0,0).
\]
We briefly recall the notation introduced earlier. For each
$\kappa\in\{\kappa_1,\kappa_2\}$ and $n\ge1$ (with $\nu=\kappa+\tfrac12$), we introduce the
linear functionals
\[
A_{\kappa,n}(v_1)
:=
\int_0^1
2 j_{\nu,n} x\,J_{\nu-1}\!\bigl(j_{\nu,n}x\bigr)\,
J_{\nu}\!\bigl(j_{\nu,n}x\bigr)\,v_1(x)\,dx\,,
\]
\[
B_{\kappa,n}(v_2)
:=
\int_0^1
j_{\nu,n} x
\Bigl(
J_{\nu}\!\bigl(j_{\nu,n}x\bigr)^2
-
J_{\nu-1}\!\bigl(j_{\nu,n}x\bigr)^2
\Bigr)\,v_2(x)\,dx\,,
\]
which define bounded linear functionals on $L^2(0,1)$.

\medskip
\noindent
Using Lemma~\ref{akns-akns_lemme_bessel_sinus} and
Remark~\ref{separation}, these linear forms admit the following
transformation--operator representations: for $n\ge1$, 
\begin{align}
	A_{\kappa,n}(v_1)
	&=-
	\frac{2}{\pi}
	\int_0^1
	\Phi_{\kappa,1}\!\bigl(j_{\nu,n}x\bigr)\,v_1(x)\,dx
	\label{eq:Aelln-Phi}
	\\
	&= -
	\frac{2}{\pi}
	\int_0^1
	\sin\!\bigl(2j_{\nu,n}x\bigr)\,
	T_\kappa^{\,1}(v_1)(x)\,dx\,,
	\label{eq:Aelln-sin}
	\\[0.4em]
	B_{\kappa,n}(v_2)
	&=
	\frac{2}{\pi}
	\int_0^1
	\Phi_{\kappa,2}\!\bigl(j_{\nu,n}x\bigr)\,v_2(x)\,dx
	\label{eq:Belln-Phi}
	\\
	&=
	\frac{2}{\pi}
	\int_0^1
	\cos\!\bigl(2j_{\nu,n}x\bigr)\,
	T_\kappa^{\,2}(v_2)(x)\,dx\,.
	\label{eq:Belln-cos}
\end{align}

\medskip\noindent
As shown in Section~\ref{subsec:decouple}, there exists a bounded isomorphism
\[
\mathcal U:\ell^2(\mathbb Z)\times\ell^2(\mathbb Z)
\longrightarrow
 \R^2 \times \ell^2(\mathbb N^*)^2 \times\ell^2(\mathbb N^*)^2
\]
such that $\mathcal U\circ S$ is block diagonal:
\[
(\mathcal U\circ S)(v_1,v_2)
=
\bigl(\mathcal M(v_2),\;\mathcal A_{\kappa_1,\kappa_2}(v_1)\,,\;
\mathcal B_{\kappa_1,\kappa_2}(v_2)\bigr)\,,
\]
where
\[
\mathcal A_{\kappa_1,\kappa_2}(v_1)
=
\bigl((A_{\kappa_1,n}(v_1))_{n\ge1},\;(A_{\kappa_2,n}(v_1))_{n\ge1}\bigr)\,,
\]
\[
\mathcal B_{\kappa_1,\kappa_2}(v_2)
=
\bigl((B_{\kappa_1,n}(v_2))_{n\ge1},\;(B_{\kappa_2,n}(v_2))_{n\ge1}\bigr)\,,
\]
\[
\mathcal M(v_2)
=
\left(-\int_0^1 x^{2\kappa_1}v_2,\;
-\int_0^1 x^{2\kappa_2}v_2\right)\,.
\]

\medskip\noindent
Since $\mathcal U$ is an isomorphism, the range of $S$ is closed if and only if
the range of $\mathcal U\circ S$ is closed. The closed--range property thus
reduces to the independent analysis of $\mathcal M$,
$\mathcal A_{\kappa_1,\kappa_2}$ and $\mathcal B_{\kappa_1,\kappa_2}\,$.

\subsection{Strategy of the proof}
\label{subsec:strategy}

We now outline the strategy used to prove the closed--range property.
For simplicity, we present the argument in the model case
$(\kappa_1,\kappa_2)=(0,1)\,$, the general case is identical.

\medskip
\noindent
We approximate the operator $\mathcal A_{\kappa_1,\kappa_2}$ by replacing the
Bessel zeros $j_{\nu,n}$ with their leading asymptotics, namely
\[
j_{\frac12,n}\sim n\pi \quad (\kappa=0)\,,
\qquad
j_{\frac32,n}\sim (n+\tfrac12)\pi \quad (\kappa=1)\,.
\]
This leads to a Fourier--type model operator whose kernels involve pure sine
functions with frequencies $2n\pi$ and $2(n+\tfrac12)\pi\,$. We show that this
model operator is injective with closed range, hence semi--Fredholm.

\medskip
\noindent
The difference between the original operator and the Fourier model operator
is compact. The proof of this fact is identical to the one given in
Appendix~B of~\cite{GGHN2025}, and we therefore omit the details.
Since the semi--Fredholm property is stable under compact perturbations,
it follows that $\mathcal A_{\kappa_1,\kappa_2}$ has closed range.

\medskip
\noindent
Applying the same argument to the family $(B_{\kappa,n})_{n\ge1}$ shows that the
block operator $(\mathcal M,\mathcal B_{\kappa_1,\kappa_2})$ also has closed range, and this proves that the differential $S$
has closed range.

\subsection{Trigonometric model}
\label{subsec:trig-model}

Using the sine representation~\eqref{eq:Aelln-sin}, we introduce a
Fourier--type model operator corresponding to $(\kappa_1,\kappa_2)=(0,1)$:
\[
\mathcal A^{(0)}_{0,1}(v_1)
=
\Bigl(
(\widetilde A_{0,n}(v_1))_{n\ge1}\,,
(\widetilde A_{1,n}(v_1))_{n\ge1}
\Bigr)\,,
\]
where
\begin{align*}
\widetilde A_{0,n}(v_1)
&=
\frac{2}{\pi}
\int_0^1
\sin(2n\pi x)\,v_1(x)\,dx\,,
\\[0.4em]
\widetilde A_{1,n}(v_1)
&=
-\frac{2}{\pi}
\int_0^1
\sin\!\bigl((2n+1)\pi x\bigr)\,
\bigl(S_{0,1} [v_1]\bigr)(x)\,dx\,.
\end{align*}
Similarly, using the cosine representation~\eqref{eq:Belln-cos}, we define
\[
\mathcal B^{(0)}_{0,1}(v_2)
=
\Bigl(
(\widetilde B_{0,n}(v_2))_{n\ge1}\,,
(\widetilde B_{1,n}(v_2))_{n\ge1}
\Bigr)\,,
\]
with
\begin{align*}
\widetilde B_{0,n}(v_2)
&=
\frac{2}{\pi}
\int_0^1
\cos(2n\pi x)\,v_2(x)\,dx\,,
\\[0.4em]
\widetilde B_{1,n}(v_2)
&=
\frac{2}{\pi}
\int_0^1
\cos\!\bigl((2n+1)\pi x\bigr)\,
\bigl(S_{0,1}[v_2]\bigr)(x)\,dx\,.
\end{align*}

\medskip
\noindent
From now on, we focus on $\mathcal A^{(0)}_{0,1}$ and, for simplicity, we
write $v$ instead of $v_1$. Using the trigonometric form above, we consider
\[
\|\mathcal A^{(0)}_{0,1}v\|_{\ell^2(\N^*)\times\ell^2(\N^*)}^2
=
\sum_{n\ge1} |\widetilde A_{0,n}(v)|^2
+
\sum_{n\ge1} |\widetilde A_{1,n}(v)|^2\,.
\]
\medskip
\noindent
The first term corresponds to the classical sine Fourier coefficients
\[
\widetilde A_{0,n}(v)
=
\frac{2}{\pi}\int_0^1 \sin(2n\pi x)\,v(x)\,dx\,.
\]
Since $\sin(2n\pi x)$ is odd with respect to $x=\tfrac12\,$, only the odd part of
$v$ contributes. By Parseval's identity,
\[
\sum_{n\ge1} |\widetilde A_{0,n}(v)|^2
=
\sum_{n\ge1} |\widetilde A_{0,n}(v_{\mathrm{odd}})|^2
=
\frac{2}{\pi^2}\,\|v_{\mathrm{odd}}\|_{L^2(0,1)}^2\,,
\]
where $v_{\mathrm{odd}}(x):=\tfrac12\bigl(v(x)-v(1-x)\bigr)$ denotes the odd
part of $v$ with respect to $x=\tfrac12\,$.

\medskip
\noindent
The second term involves the shifted sine basis and the transform $S_{0,1}$:
\[
\widetilde A_{1,n}(v)
=
-\frac{2}{\pi}\int_0^1
\sin\!\bigl((2n+1)\pi x\bigr)\,(S_{0,1}[v])(x)\,dx\,.
\]
Since $\sin((2n+1)\pi x)$ is even with respect to $x=\tfrac12\,$, only the even
part of $S_{0,1}[v]$ contributes. Including the mode $n=0\,$, Parseval's identity
yields
\[
\sum_{n\ge0} |\widetilde A_{1,n}(v)|^2
=
\frac{2}{\pi^2}\,\|(S_{0,1}[v])_{\mathrm{even}}\|_{L^2(0,1)}^2\,,
\]
where
\[
(S_{0,1}[v])_{\mathrm{even}}(x)
:=
\tfrac12\bigl((S_{0,1}[v])(x)+(S_{0,1}[v])(1-x)\bigr)\,.
\]
If we start the sum at $n=1$, we subtract the contribution of the mode
$\sin(\pi x)$, namely
\[
|\widetilde A_{1,0}(v)|^2
=
\frac{2}{\pi^2}\,
\bigl|\langle (S_{0,1}[v])_{\mathrm{even}},\,\sin(\pi x)\rangle_{L^2(0,1)}\bigr|^2\,.
\]
Let $P$ denote the orthogonal projection onto the subspace
$L^2_{\mathrm{even}}(0,1)$ (with respect to $x=\tfrac12$). Then
\[
(S_{0,1}[v])_{\mathrm{even}} = P S_{0,1}[v]\,.
\]
Therefore, starting the sum at $n=1$, Parseval's identity yields
\[
\sum_{n\ge1} |\widetilde A_{1,n}(v)|^2
=
\frac{2}{\pi^2}\,\|(S_{0,1}[v])_{\mathrm{even}} \|_{L^2(0,1)}^2
-
\frac{2}{\pi^2}\,
\bigl|\langle P S_{0,1}[v],\,\sin(\pi x)\rangle_{L^2(0,1)}\bigr|^2\,.
\]
We may rewrite the scalar product using the adjoint of $P S_{0,1}$:
\[
\langle P S_{0,1}[v],\,\sin(\pi x)\rangle_{L^2(0,1)}
=
\langle v,\,(P S_{0,1})^{*}\sin(\pi x)\rangle_{L^2(0,1)}\,.
\]
Hence
\[
\sum_{n\ge1} |\widetilde A_{1,n}(v)|^2
=
\frac{2}{\pi^2}\,\|(S_{0,1}[v])_{\mathrm{even}}\|_{L^2(0,1)}^2
-
\frac{2}{\pi^2}\,
\bigl|\langle v,\,(P S_{0,1})^{*}\sin(\pi x)\rangle_{L^2(0,1)}\bigr|^2\,.
\]

\medskip
\noindent
Combining both contributions, we obtain
\[
\|\mathcal A^{(0)}_{0,1}v\|_{\ell^2(\N^*)\times\ell^2(\N^*}^2
=
\frac{2}{\pi^2}
\Big(
\|v_{\mathrm{odd}}\|_{L^2(0,1)}^2
+
\|(S_{0,1}[v])_{\mathrm{even}}\|_{L^2(0,1)}^2
-
\bigl|\langle v,\,(P S_{0,1})^{*}\sin(\pi x)\rangle_{L^2(0,1)}\bigr|^2
\Big)\,.
\]

\medskip
\noindent
In order to exploit this identity, we focus on the even component of
$S_{0,1}[v]$ and introduce the corresponding projected transform.

\medskip
\noindent
We recall that the integral operator $S_{0,1}$ is defined by
\begin{equation}\label{eq:S01-def}
(S_{0,1}[v])(x)
:=
v(x)-2\int_x^1 \frac{v(t)}{t}\,dt\,,
\qquad x\in(0,1)\,,
\end{equation}
and we define
\[
T:=P\circ S_{0,1}
:\;
L^2_{\mathrm{even}}(0,1)
\longrightarrow
L^2_{\mathrm{even}}(0,1)\,.
\]

\begin{lemma}[A bounded left inverse for \(T\)]
\label{lem:left-inverse-PS01}
Define the operator \(L:L^2_{\mathrm{even}}(0,1)\to L^2_{\mathrm{even}}(0,1)\) by
\begin{equation}\label{eq:L-def}
(Lg)(x)
=
g(x)
-\frac{1}{x(1-x)}\int_0^x (1-2t)\,g(t)\,dt,
\qquad 0<x<1.
\end{equation}
Then:
\begin{enumerate}
\item \(L\) is a left inverse of \(T:=P\circ S_{0,1}\) on \(L^2_{\mathrm{even}}(0,1)\), i.e.
\begin{equation}\label{eq:left-inverse-identity}
L\,T v=v,
\qquad \forall v\in L^2_{\mathrm{even}}(0,1).
\end{equation}
\item The operator \(L\) is bounded on \(L^2_{\mathrm{even}}(0,1)\). More precisely,
\[
\|Lg\|_{L^2(0,1)}
\le
5\,\|g\|_{L^2(0,1)},
\qquad
\forall\, g\in L^2_{\mathrm{even}}(0,1).
\]
\end{enumerate}
\end{lemma}

\begin{proof}

\medskip
\noindent\emph{Step 1: derivation of the formula.}
By density, we may assume without loss of generality that $v\in C^1([0,1])$.
We recall that $v$ is even and set $g:=Tv=P S_{0,1}[v]$.
Using the definition of $S_{0,1}$ together with the symmetry $v(1-x)=v(x)$,
one obtains for $x\in(0,1)$
\[
g(x)=v(x)-\int_x^1 \frac{v(t)}{t}\,dt-\int_0^x \frac{v(t)}{1-t}\,dt.
\]
Differentiating gives
\[
g'(x)=v'(x)+\frac{v(x)}{x}-\frac{v(x)}{1-x}
=v'(x)+\frac{1-2x}{x(1-x)}\,v(x),
\]
hence
\[
x(1-x)v'(x)+(1-2x)v(x)=x(1-x)g'(x),
\qquad\text{i.e.}\qquad
\bigl(x(1-x)v(x)\bigr)' = x(1-x)g'(x).
\]
Integrating from $0$ to $x$ and performing one integration by parts yields
\[
x(1-x)v(x)=x(1-x)g(x)-\int_0^x (1-2t)\,g(t)\,dt,
\]
that is,
\[
v(x)=g(x)-\frac{1}{x(1-x)}\int_0^x (1-2t)\,g(t)\,dt,
\qquad 0<x<1.
\]
Since $g$ is even with respect to $\tfrac12$, the right-hand side is also even,
hence $v$ is even as well. This is precisely the formula defining the left inverse
$L$.

\medskip
\noindent\emph{Step~2: boundedness on $L^2_{\mathrm{even}}(0,1)$.}
Since $Lg$ is even with respect to $\tfrac12$, it suffices to work on $(0,\tfrac12)$:
\[
\|Lg\|_{L^2(0,1)}^2=2\|Lg\|_{L^2(0,1/2)}^2.
\]
For $0<x\le\tfrac12$,
\[
(Lg)(x)=g(x)-\frac{1}{x(1-x)}\int_0^x (1-2t)\,g(t)\,dt.
\]
Using $|1-2t|\le1$ and $1-x\ge\tfrac12$, we obtain
\[
|(Lg)(x)|\le |g(x)|+\frac{2}{x}\int_0^x|g(t)|\,dt.
\]
By Hardy's inequality on $(0,\tfrac12)$,
\[
\int_0^{1/2}\!\left(\frac{1}{x}\int_0^x |g(t)|\,dt\right)^2\!dx
\le 4\int_0^{1/2}\!|g(x)|^2dx,
\]
hence
\[
\|Lg\|_{L^2(0,1/2)}
\le \|g\|_{L^2(0,1/2)}+4\|g\|_{L^2(0,1/2)}
=5\|g\|_{L^2(0,1/2)}.
\]
Therefore $\|Lg\|_{L^2(0,1)}\le 5\|g\|_{L^2(0,1)}$, and $L$ is bounded on
$L^2_{\mathrm{even}}(0,1)$.
\end{proof}

\medskip
\noindent
We now show that the two nonnegative contributions in the right--hand side
of the following identity already control the full $L^2$--norm of $v$:
\begin{equation}\label{eq:A0-identity-missing-mode}
\|\mathcal A^{(0)}_{0,1}v\|_{\ell^2(\N^*)\times\ell^2(\N^*)}^2
=
\frac{2}{\pi^2}
\Big(
\|v_{\mathrm{odd}}\|_{L^2(0,1)}^2
+
\|(S_{0,1}[v])_{\mathrm{even}}\|_{L^2(0,1)}^2
-
|\langle v,w\rangle_{L^2(0,1)}|^2
\Big)\,,
\end{equation}
where
\[
w:=(P S_{0,1})^{*}\sin(\pi x)\in L^2(0,1)\,.
\]

\medskip
\noindent
We first prove that
\begin{equation}\label{eq:coercive-two}
\|v\|_{L^2(0,1)}^2
\;\lesssim\;
\|v_{\mathrm{odd}}\|_{L^2(0,1)}^2
+
\|(S_{0,1}[v])_{\mathrm{even}}\|_{L^2(0,1)}^2\, .
\end{equation}
Since $v=v_{\mathrm{odd}}+v_{\mathrm{even}}$, we have
\[
(S_{0,1}[v])_{\mathrm{even}}
=
(P\circ S_{0,1})v
=
(P\circ S_{0,1})v_{\mathrm{even}}
+
(P\circ S_{0,1})v_{\mathrm{odd}}\,.
\]
We recall that $T=P\circ S_{0,1}$ on $L^2_{\mathrm{even}}(0,1)$. Then
\[
T v_{\mathrm{even}}
= (P\circ S_{0,1})v_{\mathrm{even}}=
(S_{0,1}[v])_{\mathrm{even}}
-
(P\circ S_{0,1})v_{\mathrm{odd}}\,.
\]
Since $L$ is a bounded left inverse of $T$ on $L^2_{\mathrm{even}}(0,1)$, we obtain
\[
v_{\mathrm{even}}
=
L\Bigl((S_{0,1}[v])_{\mathrm{even}}
-
(P\circ S_{0,1})v_{\mathrm{odd}}\Bigr)\,,
\]
and therefore, since $P\circ S_{0,1}$ is bounded on $L^2(0,1)\,$,
\[
\|v_{\mathrm{even}}\|_{L^2(0,1)}
\le
C\Bigl(
\|(S_{0,1}[v])_{\mathrm{even}}\|_{L^2(0,1)}
+
\|v_{\mathrm{odd}}\|_{L^2(0,1)}
\Bigr)\,.
\]
Squaring and adding the odd part yields \eqref{eq:coercive-two}. Combining \eqref{eq:A0-identity-missing-mode} and \eqref{eq:coercive-two}
yields
\[
\|v\|_{L^2(0,1)}^2
\;\lesssim\;
\|\mathcal A^{(0)}_{0,1}v\|_{\ell^2(\N^*)\times\ell^2(\N^*)}^2
+
|\langle v,w\rangle_{L^2(0,1)}|^2\,.
\]
Equivalently,
\[
\|v\|_{L^2(0,1)}
\;\lesssim\;
\Big\|\bigl(\mathcal A^{(0)}_{0,1}v,\;\langle v,w\rangle_{L^2(0,1)}\bigr)\Big\|\,.
\]

\medskip
\noindent
In particular, the augmented operator 
\[
\widetilde{\mathcal A}
:=
\bigl(\mathcal A^{(0)}_{0,1},\,\langle \,\cdot\,,w\rangle_{L^2(0,1)}\bigr)
:\;L^2(0,1)\to (\ell^2(\N^*)\times\ell^2(\N^*))\times\mathbb R
\]
is injective and has closed range. Since $\langle \cdot,w\rangle$ is a
rank--one operator, $\widetilde{\mathcal A}$ is a finite--rank extension of
$\mathcal A^{(0)}_{0,1}$. Hence $\mathcal A^{(0)}_{0,1}$ is semi--Fredholm:
its kernel is at most one--dimensional and its range is closed.
This follows, for instance, from~\cite[Proposition~11.4]{Br11}.

\medskip\noindent
Now, we recall below the following
local injectivity result, which is a direct consequence of the mean value theorem
and the open mapping theorem (see, for instance,~\cite{Abraham88}, Theorem 2.5.10).

\begin{prop}[Local injectivity]\label{prop:local-injectivity}
	Let $X$ and $Y$ be Banach spaces, and let
	\[
	\mathcal{S} : U \subset X \longrightarrow Y
	\]
	be a $\mathcal C^1$ map defined on an open neighborhood $U$ of a point
	$x_0 \in X$.
	Assume that the Fr\'echet differential $d_{x_0}\mathcal{S} : X \to Y$
	is injective and has closed range.
	Then there exists a neighborhood $V \subset U$ of $x_0$ such that
	$\mathcal{S}$ is injective on $V$.
\end{prop}

\vspace{0.2cm}\noindent
Theorem~\ref{thm:main-AKNS-01} is a direct consequence of
Proposition~\ref{prop:local-injectivity},
Theorems~\ref{thm:injective-01a}, \ref{thm:injective-01c} and \ref{thm:noninj-03}, and the closedness of the range of $S$.

\begin{rem}[Other effective angular momenta]
The case $(\kappa_1,\kappa_2)=(0,2)$ is more delicate. Indeed, the asymptotics of
the corresponding Bessel zeros do not produce the half--integer phase shift
that yields the interlaced frequencies appearing in the case $(0,1)$.
As a consequence, the associated trigonometric system is no longer complete:
one only obtains a partial family (either sine or cosine), rather than a full
sine--cosine system. In particular, the argument based on the coercive identity for the
trigonometric model cannot be applied directly, since the missing family
prevents a direct control of the whole $L^2$--norm. A refined analysis is then
required to recover closed range in this case.
\end{rem}

\appendix
\section{Physical interpretation of the model: from radial Dirac operators to AKNS systems}
The AKNS  system appears in many models in Physics. We have selected below two models where the results established in the main text are relevant. This also suggests the consideration of many other questions.

\subsection{Dirac in 3D}
Following \cite{Th92} (see also \cite{AHM-1} and \cite{Serier2006}), we recall that
the MIT realization of the Dirac operator on $L^2(\mathcal B,\mathbb C^4)$ 
($\mathcal B$ is the unit ball of $\R^3$) with a radial matrix potential
\begin{subequations}
\begin{equation}
V(x):= \phi_{el}(r)  I_4 + \phi_{sc} (r) {\bf \beta} + i {\bf \beta} {\bf \alpha}\cdot e_r \phi_{am} (r)\,,
\end{equation}
where
\begin{equation}
{\bf  \beta}= \left(\begin{array}{ll} I_2& 0\\0&-I_2\end{array}\right)\,,\,
\alpha_i = \left(\begin{array}{ll} 0 & \sigma_i\\\sigma_i&0\end{array}\right)\,,{\bf \alpha} =(\alpha_1,\alpha_2,\alpha_3)\,,
\end{equation}
\begin{equation}
\sigma_1= \left(\begin{array}{ll} 0 & 1\\ 1&0\end{array}\right)\,,\,
\sigma_2= \left(\begin{array}{ll} 0 & -i\\i &0\end{array}\right)\,,\,
\sigma_3= \left(\begin{array}{ll} 1 & 0\\ 0 & -1\end{array}\right)\,,\,
\end{equation}
the $\sigma_i$ are the Pauli matrices,
\begin{equation}
e_r:= {\bf x}/r\,,
\end{equation}
and 
$\phi_{el}\,$,$\phi_{sc}$ and $\phi_{am}$ are radial potentials with a physical interpretation.
\end{subequations}\\
Although the case $\phi_{el}$ is interesting (one can find in \cite{Th92} the analysis of the Coulomb case), we are concerned in this article with the case when $\phi_{el}=0$\,, and use in the main text the notation 
$\phi_{sc}=p$ and $\phi_{am}=q\,$. Notice that, when $\phi_{el}$ is not $0$, it is known from \cite{lesa-dirac} (see also the discussion in the introduction in \cite{AHM}) that the inverse problem is ill posed
 for the AKNS system already when $\kappa=0\,$. Theorem 4.14 in \cite{Th92} states that
the Dirac operator 
\begin{subequations}\label{defDir}
\begin{equation}
\mathbb D_V:= \mathbb D_0 + V\,,
\end{equation}
with ($m$ being the mass)
\begin{equation}
\mathbb D_0 = \sum_i \alpha_i D_{x_i} + {\bf \beta} m 
\end{equation}
\end{subequations}
is unitary equivalent to the direct sum of the so-called "partial wave" Dirac operators $h_{m_j,\kappa_j}$
$$
\bigoplus_{j=\frac 12,\frac 32,\cdots}^{+\infty} \quad \bigoplus_{m_j=-j}^j  \quad \bigoplus _{\kappa_j=\pm (j+\frac 12)} h_{m_j,\kappa_j}
$$
where, in the basis $\{\Phi^{+}_{m_j,\kappa_j}, \Phi^{-}_{m_j,\kappa_j}\}\,$ (see (4.110)-(4.116) in \cite{Th92}),
$h_{m_j,\kappa_j}$ is the operator $H_{\kappa_j}$ with a suitable boundary condition at $r=1\,$.

\noindent
Notice that in this decomposition we only meet (up to unitary equivalence) the Dirac operators $H_\kappa$ on $L^2(0,1)$ for $\kappa \in \mathbb Z \setminus\{0\}\,$.
Here $\mathbb Z \setminus\{0\}$ is interpreted as the eigenvalues of some selfadjoint operator $K$
on $L^2(S^2,\mathbb C^4)\,$, where $S^2$ is the two dimensional unit sphere in $\mathbb R^3$.
We emphasize that $\kappa$ is not the  angular momentum as sometimes wrongly written
(for example in \cite{AHM}).

\noindent
Notice also that in the  Subsection 4.6.6 in  \cite{Th92}  only the case in  $(0,+\infty)$ is considered but this does not change the "tangential" decomposition of $L^2(S^2,\mathbb C^4)$. \\ 

\noindent
 Hence we have to analyze more carefully the possible boundary conditions by coming back to the problem for the unit ball in $\mathbb R^3$. 
According to \cite{AMSV},
the generalized MIT condition in a domain $\Omega$ is given by
$$
\varphi =\frac i2 (\lambda_e-\lambda_s \beta) ({\bf \alpha}\cdot \nu) \varphi \mbox{ on } \partial \Omega\,,
$$
with 
$$
\lambda_e^2-\lambda_s^2=-4\,.
$$
Notice that the standard MIT model corresponds with $\lambda_e=0$, $\lambda_s=\pm 2\,$.\\

\noindent
In the case of the ball and for the standard case, we get
$$
\varphi =-i {\bf \beta} ({\bf \alpha}\cdot e_r) \varphi \mbox{ on } S^2\,.
$$
Using Lemma 4.13 in \cite{Th92}, the operators ${\bf \beta}$ and ${\bf \alpha}\cdot e_r$  respect the decomposition and, with respect to the basis 
$\{\Phi^{+}_{m_j,\kappa_j}, \Phi^{-}_{m_j,\kappa_j}\}$,  are represented by the $2\times 2$ matrices
$$
\beta_{m_j,\kappa_j}=\sigma_3\quad \text{and} \quad  (-i{\bf \alpha}\cdot e_r)= -i \sigma_2\,.
$$

\noindent
The boundary condition consequently reads
$$
(f^+,f^-)^{T} = \sigma_1  (f^+,f^-)^{T}\,,\mbox{ for } r=1\,,
$$
or
$$
f^+(1)+ f^{-} (1)=0\,.
$$
This corresponds in the AKNS notation to $\beta=\frac \pi 4$.\\

\noindent
Let us consider now the general MIT condition. We get
$$
(f^+,f^-)^{T} = \frac 12 \left( \begin{array}{ll} 0& \lambda_e - \lambda_s\\ \lambda_s +  \lambda_e & 0 \end{array} \right) (f^+,f^-)^{T}\,,\mbox{ for } r=1\,,
$$
which reads
$$
f^+(1)= \frac 12 (\lambda_e- \lambda_s)f^{-} (1)\,.
$$
If we take $\lambda_s$ and $\lambda_e$ of opposite sign and take $\lambda_e \rightarrow +\infty\,$, we get at the limit
$$
f^- (1)=0\,,
$$
which corresponds in the AKNS formalism to $\theta_2=0$. This limit is analyzed in \cite{AMSV} and this justifies to consider this limiting case also called Zig-Zag model.

\noindent
More directly, this model  is analyzed in \cite{Ho} who refers to \cite{Sch}. Other properties 
for the radial Dirac operator are considered in \cite{ALMR,GL}.

\subsection{Dirac in 2D with Aharonov-Bohm potential}
 
It is natural to consider the same problem in dimension $2$.  Here we refer to another section in \cite{Th92} or to \cite{AMSV}. Here we naturally get 
an AKNS family with $\kappa \in \tfrac{1}{2} + \mathbb{Z}$. In this case, the Dirac operator is a $2\times 2$ system. The free Dirac operator
 reads 
 $$
 \mathbb D_0 = \sigma_1 D_{x_1} + \sigma_2 D_{x_2}\,,
 $$
 and we can add a potential in the form
 $$
 V =\left(\begin{array}{ll} -q&p\\p&q\end{array}\right)\,,
 $$
 as it appears in the AKNS system. \\
 The description of the decomposition in the radial case is simpler than in the $3D$ case and we have just to consider the polar coordinates. This is precisely described in Thaller's book (\cite{Th92}, Subsection 7.3.3) 
  but we have to explain two points which are not present there. 
  First, since we are interested in the case of the disk, we have to describe what would be the boundary condition. This is for example discussed for general domains with $C^2$ boundary in \cite{BFSV} (see also references therein), the simplest  conditions becoming simply (Zig-Zag model):
  $$
  (\gamma v_1)_{|\partial \Omega} = 0\,,
  $$
  or
  $$
  (\gamma v_2)_{|\partial \Omega} = 0\,,
  $$
  where $\gamma$ denotes the trace operator.\\
  In the reduction using the decomposition in \cite{Th92} we get the boundary condition $Y_2(0)=0\,$.  Other conditions could be discussed. According to Lemma 2.3 in \cite{BFSV}, the general condition reads
  $$
  (\gamma v_2)_{|\partial \Omega} = \frac{1-\sin \eta}{\cos \eta} t(s) (\gamma v_2)_{|\partial \Omega}\,.
  $$
  (where, for $s\in \partial \Omega$, $t(s):=t_1(s)+ i t_2(s)$, 
   $(t_1(s),t_2(s)$ is the tangent vector to $\partial \Omega$ at $s$,  see p.2, line -3 in \cite{BFSV}).
  In Theorem 1.1 in \cite{BFSV}, it is assumed that  (see Remark 2)  $\cos \eta \neq 0$ for having a regular self-adjoint problem with compact resolvent. Nevertheless in the Zig-Zag case, one can also define a natural selfadjoint extension. $0$ seems to belong to the essential spectrum. The results are described in the recent paper \cite{DMS} which refers to a paper by K. Schmidt \cite{Sch}.
  The corresponding family of the AKNS operators is indexed by $\kappa =\pm (1/2, 3/2,\cdots,)$  with boundary condition at $r=1$ given by $\theta_2=0\,$.\\
  Unfortunately, we do not know how to treat this problem when the $\kappa$ are not in $\mathbb Z$.\\
  As already observed in \cite{Th92}, one can perform the same decomposition in the case when the magnetic potential $A_\phi(r)e_\phi$ (with $e_\phi= \frac 1r (-x_2,x_1)$) corresponds to a radial magnetic field $B(r)$. The decomposition leads simply to replace in the definition of the AKNS system $\frac{d}{dr} - \frac{\kappa}{r}$ by
  $\frac{d}{dr} - \frac{\kappa}{r}+A_\phi(r)$ (see Formula (7.103) in \cite{Th92}). \\
  We want to consider $A_\phi(r) =\frac{\alpha}{r}\, $. The formal part of the decomposition still works but the regularity assumption done in \cite{Th92} is not satisfied since the corresponding magnetic field is 
  $2\pi \alpha \delta_0$ where $\delta_0$ denotes the Dirac measure at the origin. As usual we can reduce the analysis to $\alpha\in [0,1)\,$.  The case $\alpha=0$ being the previously discussed case without magnetic potential, it remains to consider $\alpha \in (0,1)$.
  Hence we have to define the domain of this magnetic  Dirac operator
   in this so called Aharonov-Bohm situation. This is fortunately discussed
   in the literature (\cite{Pe,Ta}). The authors classify in the case of $\mathbb R^2$  all the possible
    selfadjoint extensions of the minimal realization starting from $C_0^\infty(\mathbb R^2\setminus \{0\};\mathbb C^2)$. As described in \cite{Ta}, we choose the condition corresponding  to the parameter $\zeta=0$ and (taking also account of the boundary condition, which is not present in Tamura's paper \cite{Ta}) the domain is
    $$
    D(\mathbb D_{\alpha,V}):= \{u= (u_1,u_2)\in L^2(\Omega)^2, D_\alpha u \in L^2(\Omega)^2\,,\, \lim_{|x|\rightarrow 0} |x|^{1-\alpha} e^{-i\theta} u_2(x)=0\,,\,  (\gamma u_2)_{\partial \Omega} = 0\}\,.
    $$
    In the case of the unit disk $\Omega =B^1$\,, we get the AKNS system in $(0,1)$ with
    $\kappa$ replaced by $\kappa_\alpha=\kappa +\alpha\,$.\\
    When $\alpha =\frac 12$ we get a sequence of integers in $\mathbb Z$
     for which the analysis of the main text is relevant.
 
  \subsection{Open problems}
  Notice that more generally, it is interesting to consider the AKNS systems
   without to assume that $\kappa$ is an integer and with any boundary condition at the origin (for the relevant $\kappa$) and at $r=1$\,.\\
  In view of the application to the two-dimensional Dirac operator,
we note in particular that Theorem~\ref{thm:muntz-AKNS} remains valid
even when the parameter $\kappa_k$ is not assumed to be an integer. 
   It could also be interesting to look at the case with a mass $m \neq 0$\,. At the level of the AKNS system this seems to correspond  to the study of  a model where the variation of $\phi_{el}$ is considered and the other potentials are $0$. In this direction, we refer to \cite{Kiss}, where an Ambarzumian-type theorem is established for Dirac operators. This result provides a uniqueness statement at the unperturbed point, showing that the vanishing of the potential $\phi_{el}$ is uniquely determined by the corresponding spectral data.

\noindent
Finally, in light of \cite{AHM}, it is natural to investigate the corresponding
Schr\"odinger problems with Robin boundary conditions.
This stems from the structural link between Dirac and Schr\"odinger frameworks:
in the Dirac setting introduced in \cite{AHM}, when the scalar potential $\phi_{sc}=0$,
the system reduces to a second-order Schr\"odinger (Bessel-type) equation,
and the boundary conditions naturally translate into Robin-type
conditions for the associated Schr\"odinger operator.

	\vspace{0.5cm}
	
	\noindent \footnotesize{
		
		\noindent Laboratoire de Math\'ematiques Jean Leray, UMR CNRS 6629. Nantes Universit\'e  F-44000 Nantes  \\
		\emph{Email adress}: damien.gobin@univ-nantes.fr \\
		
		\noindent Laboratoire de Math\'ematiques Jean Leray, UMR CNRS 6629. Nantes Universit\'e  F-44000 Nantes  \\
		\emph{Email adress}: benoit.grebert@univ-nantes.fr \\
		
		\noindent Laboratoire de Math\'ematiques Jean Leray, UMR CNRS 6629. Nantes Universit\'e  F-44000 Nantes  \\
		\emph{Email adress}: Bernard.Helffer@univ-nantes.fr \\
		
		\noindent Laboratoire de Math\'ematiques Jean Leray, UMR CNRS 6629. Nantes Universit\'e  F-44000 Nantes \\
		\emph{Email adress}: francois.nicoleau@univ-nantes.fr \\
	}
	
\end{document}